\documentclass{amsproc}
\usepackage[matrix,arrow,tips,curve]{xy}
\usepackage[english]{babel}
\usepackage[leqno]{amsmath}
\usepackage{amssymb,amsthm}
\usepackage{amscd}
\usepackage{enumerate}
\usepackage{epsfig}
\usepackage{layout}
\usepackage{fullpage}

\input{xy}
\xyoption{all}

\newtheorem{thm}{Theorem}[subsection]
\newtheorem{lemma}[thm]{Lemma}
\newtheorem{lemmadefi}[thm]{Lemma - Definition}

\newtheorem{factdefi}[thm]{Fact - Definition}
\newtheorem{prop}[thm]{Proposition}

\newtheorem{cor}[thm]{Corollary}

\newtheorem{fact}[thm]{Fact}

\theoremstyle{remark}
\newtheorem{remark}[thm]{Remark}
\theoremstyle{definition}
\newtheorem{defi}[thm]{Definition}

\newtheorem{example}[thm]{Example}

\numberwithin{equation}{section}
\newcommand{\Q}{\mathbb{Q}}
\newcommand{\Z}{\mathbb{Z}}
\newcommand{\R}{\mathbb{R}}

\newcommand{\CC}{\mathbb{C}}

\newcommand{\C}{\mathcal{C}}
\renewcommand{\P}{\mathbb{P}}

\newcommand{\ov}{\overline}
\newcommand{\un}{\underline}

\newcommand{\wt}{\widetilde}

\newcommand{\Out}{\operatorname{Out}}
\newcommand{\val}{\operatorname{val}}
\newcommand{\Aut}{\operatorname{Aut}}
\newcommand{\Inn}{\operatorname{Inn}}
\newcommand{\supp}{\operatorname{supp}}

\newcommand{\GL}{\operatorname{GL}}
\newcommand{\Sp}{\operatorname{Sp}}

\newcommand{\Tgp}{\ensuremath{\mathcal T^{\text{tr},p}_g}}
\newcommand{\Tg}{\ensuremath{\mathcal T^{\text{tr}}_g}}
\newcommand{\Atrg}{\ensuremath{A^{\text{tr}}_g}}

\newcommand{\Mtrg}{\ensuremath{M^{\text{tr}}_g}}

\newcommand{\hooklongrightarrow}{\lhook\joinrel\longrightarrow}

\newcommand{\Mp}{\ensuremath{M^{\text{tr}, p}}_g}

\newcommand{\Del}{{\rm Del}}

\newcommand{\prin}{\sigma_{{\rm prin}}}

\renewcommand{\O}{\Omega_g}
\newcommand{\Ort}{\Omega_g^{\rm rt}}

\newcommand{\Null}{{\rm Null}}

\newcommand{\Per}{\Sigma_{\operatorname{P}}}
\newcommand{\Voro}{\Sigma_{\operatorname{V}}}

\newcommand{\ra}{\rightarrow}
\newcommand{\comment}[1]{} 
\newcommand{\al}{\alpha}
\newcommand{\G}{\Gamma}

\newcommand{\Hg}{{}^{\Sigma}\mathbb{H}_g^{tr}}
\newcommand{\Hgp}{{}^{\Sigma}\mathbb{H}_g^{tr,p}}

\newcommand{\Ag}{{}^{\Sigma}\!{A}_g^{tr}}
\newcommand{\Agp}{{}^{\Sigma}\!{A}_g^{tr,p}}

\newcommand{\Pg}{{}^{\Sigma}{\mathcal P}_g^{tr}}
\newcommand{\Pgp}{{\mathcal P}_g^{tr,p}}

\newcommand{\tg}{{}^{\Sigma}t_g^{tr}}
\newcommand{\tgp}{t_g^{tr,p}}

\newcommand{\Ab}{{\mathfrak A}}

\newcommand{\Tei}{{\mathcal T}_g}
\newcommand{\Sie}{{\mathbb H}_g}
\newcommand{\Peri}{\mathcal{P}_g}

\begin{document}

\title{Tropical Teichm\"uller and Siegel spaces}

\author[Chan]{Melody Chan}
\address{Department of Mathematics, Harvard University, One Oxford Street, Cambridge MA 02138, USA} \email{mtchan@math.harvard.edu}

\author[Melo]{Margarida Melo}
\address{Departamento de Matem\'atica, Universidade de Coimbra,
Largo D. Dinis, Apartado 3008, 3001 Coimbra, Portugal}
\email{mmelo@mat.uc.pt}

\author[Viviani]{Filippo Viviani}
\address{ Dipartimento di Matematica,
Universit\`a Roma Tre,
Largo S. Leonardo Murialdo 1,
00146 Roma, Italy}
\email{viviani@mat.uniroma3.it}

\maketitle

\begin{abstract}
In this paper, we present a unified study of the moduli space of tropical curves and Outer space which we link via period maps to the moduli space of tropical abelian varieties and the space of positive definite quadratic forms. Our aim is to exhibit Outer space and the space of positive definite quadratic forms as analogues of Teichm\"uller space and Siegel space, respectively, in tropical geometry. All these spaces and the maps among them are described within the category of ideal stacky fans, which we describe in detail.
\end{abstract}

\tableofcontents

\section{Introduction}

The aim of this paper is to develop a tropical analogue of Teichm\"uller space, of Siegel space, and of the period map from the former space to the latter one.  Indeed, in some sense, this analogy has been known in geometric group theory well before the advent of tropical geometry; however, we hope to convince the reader that the tropical viewpoint can bring some new insight into the picture.

\vspace{0,1cm}

For the reader's convenience, we start by first reviewing the classical theory and then explaining the tropical analogy that we want to pursue here.

\subsection{Classical theory: Teichm\"uller, Siegel and the period map}
In this subsection, we give a very short overview of the classical period map from the Teichm\"uller space to the Siegel space,
referring the reader to \cite[Chap. XV]{GAC2} and \cite[Chap. 8, Chap, 11]{BL} for more details and references.

Fix a connected, compact, orientable topological surface $S_g$ of genus $g\geq 2$.\footnote{This assumption on $g$ is made only for simplicity. Indeed, everything we are going to say is trivial for $g=0$ and it can be easily adapted to the case  $g=1$, where everything is much more simple.} Let $x_0\in S_g$ and set $\Pi_g:=\pi_1(S_g,x_0)$. The group $H_1(S_g,\Z)=\Pi_g^{\rm ab}:=\Pi_g/[\Pi_g,\Pi_g]\cong\Z^{2g}$ has a symplectic structure given by the intersection pairing  $(\cdot,\cdot)$ on $S_g$. It is possible to chose a basis of $H_1(S_g,\Z)$ with respect to which $(\cdot,\cdot)$ is given by the standard symplectic form
\begin{equation*}\label{E:sym-form}
 Q=\begin{pmatrix}0 & {\mathbb I}_g\\ -{\mathbb I}_g & 0 \end{pmatrix}.
 \end{equation*}

The {\bf {\em Teichm\"uller space}} $\Tei$ is the fine moduli space of marked Riemann surfaces of genus $g$, i.e. pairs $(C,h)$ where $C$ is a  Riemann surface of genus $g$ and $h\colon S_g\stackrel{\cong}{\longrightarrow} C$ is the homotopy class of an orientation-preserving homeomorphism, called a marking of $C$.
 The space $\Tei$ is a complex manifold of dimension $3g-3$ which is homeomorphic (although not biholomorphic) to the unit ball in $\CC^{3g-3}$ (see \cite[Chap. XV, Thm. (4.1)]{GAC2}). In particular, $\Tei$ is a contractible space.

Consider now the outer automorphism group $\Out(\Pi_g):=\Aut(\Pi_g)/\Inn(\Pi_g)$ of $\Pi_g$. Since $S_g$ is an Eilenberg-MacLane $K(\Pi_g, 1)$-space, every element $\alpha\in \Out(\Pi_g)$ is induced by a homotopy equivalence $S_g\ra S_g$ which, by the Dehn-Nielsen-Baer theorem, we may choose to be a homeomorphism $\alpha_{S_g}:S_g\stackrel{\cong}{\longrightarrow} S_g$.  The map $\alpha_{S_g}$ is unique up to homotopy and we call it the geometric realization of $\alpha$.  The {\em mapping class group} $\Gamma_g$ is the index two subgroup of $\Out(F_g)$ consisting of the elements $\alpha\in \Out(F_g)$ such that the geometric realization $\alpha_{S_g}$ is orientation-preserving.

The mapping class group acts on the Teichm\"uller space $\Tei$ by changing the marking. More precisely, an element $\alpha\in \Gamma_g$ acts on $\Tei$ by sending $(C,h)\in \Tei$ to
$$(C,h)\circ \alpha:=(C, h\circ \alpha_{S_g}).$$
The action of $\Gamma_g$ is properly discontinuous (see \cite[p. 452]{GAC2}) and  the quotient $M_g:=\Tei/\Gamma_g$ is a complex quasi-projective variety which turns out to be
the coarse moduli space of Riemann surfaces of genus $g$.

The {\bf {\em Siegel space}} $\Sie$ is the fine moduli space of marked principally polarized (p.p.~for short) abelian varieties of dimension $g$, i.e.~triples $(V/\Lambda,E,\phi)$ where  $V$ is a complex $g$-dimensional vector space, $\Lambda\subset V$ is a full-dimensional lattice (so that $V/\Lambda$ is a complex torus of dimension $g$), $E$ is a principal polarization on the torus $V/\Lambda$  and  $\phi\colon(\Z^{2g},Q)\stackrel{\cong}{\longrightarrow} (\Lambda, E)$ is a symplectic isomorphism (called a marking of the principally polarized torus $(V/\Lambda, E)$).
Recall that a principal polarization on the complex torus $V/\Lambda$ is a symplectic form $E:V\times V\to \R$ such that:
\begin{enumerate}[(i)]
\item $E(\Lambda,\Lambda)\subset \Z$;
\item $E(iv,iw)=E(v,w)$ for any $v,w\in V$;
\item There exists a symplectic isomorphism $(\Z^{2g},Q)\stackrel{\cong}{\longrightarrow} (\Lambda, E)$.
\end{enumerate}
The space $\Sie$ is a complex  manifold of dimension $\binom{g+1}{2}$ which is moreover contractible (see \cite[Sec. 8.1]{BL}).

The {\em symplectic group} $\Sp_{2g}(\Z):=\Aut(\Z^{2g}, Q)$ acts on $\Sie$ by changing the marking. More precisely, an element $\alpha\in \Sp_{2g}(\Z)$ acts on $\Sie$ by sending $(V/\Lambda,E,\phi)\in \Sie$ into
$$(V/\Lambda,E,\phi)\circ \alpha:=(V/\Lambda, E, \phi \circ \alpha).$$
The action of $\Sp_{2g}(\Z)$ is properly discontinuous (see \cite[Prop. 8.2.5]{BL}) and  the quotient $A_g:=\Sie/\Sp_{2g}(\Z)$ is a complex quasi-projective variety that is
the coarse moduli space of p.p.~abelian varieties of dimension $g$ (see \cite[Thm. 8.2.6]{BL}).

Given a Riemann surface $C$ of genus $g$, denote by $\Omega_C^1$ the sheaf of K\"ahler differentials on $C$; thus elements of  $H^0(C,\Omega_C^1)$ are holomorphic $1$-forms on $C$.
The {\em Jacobian} $J(C)$ of  $C$ is the p.p.~abelian variety of dimension $g$ given by the complex torus $H^0(C,\Omega_C^1)^*/H_1(C,\Z)$ (where
the injective map $H_1(C,\Z)\hookrightarrow H^0(C,\Omega_C^1)^*$ is given by the integration of holomorphic $1$-forms along $1$-cycles) together with the principal polarization $E_C$ coming from the intersection product  on $H_1(C,\Z)$. Any marking $h:S_g\stackrel{\cong}{\longrightarrow} C$ of $C$ gives rise to a marking $\phi_h:(\Z^{2g},Q)=(H_1(S_g,\Z),(\cdot, \cdot))\stackrel{\cong}{\longrightarrow} (H_1(C,\Z), Q_C)$ of $J(C)$. Associating to a marked Riemann surface of genus $g$ its marked Jacobian we get the following holomorphic map, called the {\bf {\em period map}}:
\begin{equation*}
\begin{aligned}
\Peri: \Tei & \longrightarrow \Sie \\
(C,h) & \mapsto (J(C),\phi_h).
\end{aligned}
\end{equation*}
Consider now the group homomorphism
\begin{equation*}\label{E:map-chi}
\begin{aligned}
\chi:\Gamma_g& \longrightarrow \Sp_{2g}(\Z) \\
\alpha &\mapsto (\alpha_{S_g})_*
\end{aligned}
\end{equation*}
where $(\alpha_{S_g})_*\colon H_1(S_g,\Z)\stackrel{\cong}{\longrightarrow} H_1(S_g,\Z)$ is the symplectic automorphism induced by the geometric realization $\alpha_{S_g}$ of $\alpha$. The map $\chi$ is indeed surjective (see \cite[p. 460]{GAC2}). It is easily checked that the period map $\Peri$ is equivariant with respect to the homomorphism $\chi$ and to the actions of $\Gamma_g$ on $\Tei$ and of $\Sp_{2g}(\Z)$ on $\Sie$ that we described above. Therefore we get the following commutative diagram (in the category of complex analytic spaces)
\begin{equation}\label{E:diag-per-Tor}
\xymatrix{
\Tei \ar[r]^{\Peri} \ar[d] & \Sie \ar[d] \\
M_g\ar[r]^{t_g} & A_g
}
\end{equation}
where $t_g$ is the algebraic morphism, called the {\em Torelli morphism}, given by:
\begin{equation*}
\begin{aligned}
t_g: M_g & \longrightarrow A_g \\
C & \mapsto J(C).
\end{aligned}
\end{equation*}

\vspace{0,1cm}

After this quick review of the classical theory, we can now explain the tropical analogues of the above spaces and of the period map.

\subsection{Tropical Teichm\"uller space}

{\em Pure tropical curves}\footnote{These curves correspond to compact tropical curves up to tropical modifications in the terminology of \cite{mz}. Our definition of tropical curves is the slightly more general definition proposed in \cite{BMV}, see below. Therefore, we call this restricted subclass ''pure" tropical curves, a term which was introduced by L. Caporaso in \cite{capsurvey}.}, i.e. compact tropical manifolds of dimension $1$, are given by metric graphs,
as shown in the breakthrough paper \cite{mz} of Mikhalkin-Zharkov. Note that a graph is an Eilenberg-MacLane $K(F_g,1)$-space (for some natural number $g$, called the genus of the graph), where $F_g$ is the free group on $g$ generators. Therefore, in the tropical picture, Riemann surfaces of genus $g$ are replaced by metric graphs of genus $g$; the fundamental group $\Pi_g$ of a Riemann surface of genus $g$ is replaced by the fundamental group $F_g$ of a graph of genus $g$; and the mapping class group $\Gamma_g$ is replaced by the outer automorphism group $\Out(F_g)$ of $F_g$ (note that there are no orientation-preserving restrictions in the tropical world).
As in the classical case, in order to define a marking of a tropical curve, we fix a graph of genus $g$, say the rose with $g$ petals (i.e.~the graph which has a unique vertex and $g$ loops attached to it) which we denote by $R_g$, and we define a marking of a tropical curve $C=(\Gamma, l)$ of genus $g$ (where $\Gamma$ is the graph of genus $g$ and $l$ is the length function on the edges) to be a homotopy equivalence $h:R_g\to \Gamma$, up to an isometry of the metric graph $(\Gamma,l)$; see Definition \ref{D:marked-graph} for the precise definition.
This analogy has also been pointed out by L. Caporaso in \cite[Sec. 5]{capsurvey}.

Following \cite{BMV} (which was inspired by the analogy between the moduli space of tropical curves and the moduli space of Deligne-Mumford stable curves) it is convenient for our purposes to enlarge the class of pure tropical curves by allowing (vertex)-weighted graphs. Therefore, a {\em tropical curve} will be throughout this paper a metric weighted graph $(\Gamma,w,l)$ satisfying a natural stability condition; we refer the reader to \S\ref{ss:tg}, where we also define a marking of an arbitrary tropical curve.

In view of the previous definitions, it is now clear that the analogue of the Teichm\"uller space should be the moduli space of marked metric graphs. Indeed, such a moduli space, usually denoted by $X_g$ and dubbed {\bf {\em Outer space}} by P. Shalen, was constructed, in the celebrated work of Culler-Vogtmann \cite{cullervogtmann}\footnote{In loc. cit., the authors consider only marked metric graphs of genus $g$ with total length equal to one; however this is not very natural for our purpose, so we never normalize the total length of a metric graph.}, as a fan inside an infinite dimensional vector space; we review the construction of $X_g$ in \S \ref{S:Outerspace}.
Moreover, the group $\Out(F_g)$ acts naturally on $X_g$ by changing the marking, and this action is known to be properly discontinuous. This action of $\Out(F_g)$ on $X_g$ has been successfully used to reveal some of the features of this very interesting group (which was the original purpose of geometric group theorists in studying $X_g$); we refer the reader to the survey papers \cite{vogtmann}, \cite{Bes} and \cite{vogICM} for an update of the known results.

Our approach to the tropical Teichm\"uller theory is slightly different from the one used by Culler-Vogtmann in the definition of Outer Space. We define the {\bf {\em (pure) tropical Teichm\"uller space}}, denoted by $\Tg$ (resp.~$\Tgp$) as an abstract (i.e.~not embedded) topological space by gluing together rational polyhedral cones (resp.~ideal rational polyhedral cones, i.e.~rational polyhedral cones with some faces removed) parametrizing marked (pure) tropical curves having fixed underlying marked (weighted) graph; see Definitions \ref{Tgp} and \ref{Tg}. In fact $\Tgp$ is homeomorphic to $X_g$, which amounts to saying that $X_g$ has the simplicial topology: this result is certainly well-known to the experts in geometric group theory, and a proof can be found in \cite{guirardel-levitt}. The space $\Tg$ can be regarded as a bordification (or partial compactification) of $\Tgp\cong X_g$ and, indeed, 
it provides a modular description of the simplicial closure of $X_g$.
It would be interesting to compare $\Tg$ with the bordification of $X_g$ constructed by M. Bestvina and M. Feighn in \cite{BF} (see \S\ref{S:open}).

Topological spaces obtained by gluing together rational polyhedral (ideal) cones via lattice-preserving linear maps are called {\em (ideal) stacky fans} in this paper. The stacky fans previously introduced in \cite{BMV} and \cite{chan} are special cases of the more general definition presented here. Section \ref{s:sf} of this paper is devoted to the study of (ideal) stacky fans: we prove that they have nice topological properties (e.g., we prove that they are always Hausdorff and, under some mild conditions, also locally compact, locally path-connected, second countable and metrizable; see Corollary \ref{c:all-hausdorff} and Proposition \ref{p:sf-hausdorff}) and we define morphisms of ideal stacky fans (see Definition \ref{d:sfm}). Moreover, we introduce admissible actions of groups on (ideal) stacky fans (roughly, those actions sending ideal cones into ideal cones via lattice-preserving linear maps; see Definition \ref{d:admissible-action}) and we show that quotients by admissible actions do exist in the category of (ideal) stacky fans (see Propositions \ref{p:isstackyfan} and \ref{p:strat-is-global}). Indeed, as it will be clear in a moment, it is this last property that makes the category of ideal stacky fans particularly suited for the purposes of this paper.

The tropical Teichm\"uller space $\Tg$ and its open subset $\Tgp$, the pure tropical Teichm\"uller space, are indeed ideal stacky fans (see Propositions \ref{p:tgp-is-isf} and \ref{p:tg-is-isf}). The group $\Out(F_g)$ acts on $\Tg$ and on $\Tgp$ by changing the marking of the (pure) tropical curves (see \S\ref{outergroup}) and we show that this action is admissible in Proposition \ref{p:outer-is-admissible}. The quotients $\Mtrg:=\Tg/\Out(F_g)$ and $\Mp:=\Tgp/\Out(F_g)\cong X_g/\Out(F_g)$ (which exist in the category of ideal stacky fans by what was said before) are moduli spaces for tropical curves (resp. pure tropical curves) of genus $g$; they coincide indeed with the moduli spaces first introduced in \cite{BMV} and further studied in \cite{chan} and \cite{capsurvey} (see \S\ref{s:mtrg}). Note also that we have natural maps of ideal stacky fans $\Tg\to \Mtrg$ and $\Tgp\to \Mp$ which correspond to forgetting the marking of the tropical curves (resp. pure tropical curves).

\subsection{Tropical Siegel space}

Following \cite{BMV} and slightly generalizing \cite{mz}, we define a {\em tropical principally polarized (=p.p.) abelian variety} of dimension $g$ to be a  pair $(V/\Lambda,Q)$ consisting of a $g$-dimensional real torus $V/\Lambda$ and a positive semi-definite quadratic form $Q$ on $V$ which is rational with respect to $\Lambda$, i.e. such that its null space is defined over $\Lambda\otimes_{\Z} \Q$. We say that a tropical p.p. abelian variety $(V/\Lambda, Q)$ is {\em pure} if $Q$ is positive definite\footnote{In the paper \cite{mz}, only pure tropical p.p. abelian varieties are considered and they are simply called tropical p.p. abelian varieties.}.  A marking of a tropical p.p. abelian variety $(V/\Lambda, Q)$ is an isomorphism of real tori $\phi:\R^g/\Z^g\stackrel{\cong}{\longrightarrow} V/\Lambda$, or equivalently a linear isomorphism from $\R^g$ onto $V$ sending $\Z^g$ isomorphically onto $\Lambda$.

Therefore, in the tropical setting,  complex tori are replaced by real tori, the symplectic form giving the polarization on a complex torus is replaced by a rational positive semi-definite quadratic form, and the isomorphism of symplectic lattices giving the marking is replaced by an isomorphism of lattices. In particular, the symplectic group $\Sp_{2g}(\Z)$ is replaced by the general linear group $\GL_g(\Z)$.

Marked (pure) tropical p.p. abelian varieties of dimension $g$ are naturally parame\-trized by the cone $\Ort\subset \R^{\binom{g+1}{2}}$ (resp. $\O\subset \R^{\binom{g+1}{2}}$) of rational positive semi-definite (resp. positive definite) quadratic forms on $\R^g$. Note that $\O$ is an open cone in $\R^{\binom{g+1}{2}}$ and $\Ort$ is a subcone of the closure $\ov{\O}$ of $\O$ inside $\R^{\binom{g+1}{2}}$, which consists of all the positive semi-definite quadratic forms.

In order to put an ideal stacky fan structure on $\O$ and on $\Ort$, we need to choose an {\em admissible  decomposition} $\Sigma$ of $\Ort$, which, roughly speaking, consists of a fan of rational polyhedral cones whose support is $\Ort$ and such that $\GL_g(\Z)$ acts on $\Sigma$ with finitely many orbits (see Definition \ref{decompo}). There are few known examples of such admissible decompositions of $\Ort$: in \S\ref{S:exa-deco} we review the definition and main properties of two of them, namely the perfect cone decomposition $\Sigma_P$ (or first Voronoi decomposition) and the 2nd Voronoi decomposition $\Sigma_V$, both introduced by Voronoi in \cite{Vor}.
Given any such admissible decomposition $\Sigma$ of $\Ort$, we can define a stacky fan $\Hg$, which we call the {\bf {\em tropical Siegel space}} associated to $\Sigma$, by gluing together the cones of $\Sigma$ using the lattice-preserving linear maps given by the natural inclusions of cones of $\Sigma$ (see Definition \ref{D:trop-Sieg}). The tropical Siegel space $\Hg$ contains an open subset $\Hgp\subset \Hg$, called the {\bf {\em pure tropical Siegel space}} (associated with $\Sigma$), which is the ideal stacky fan obtained from the stacky fan $\Hg$ by removing the cones that are entirely contained in the boundary $\Ort\setminus \O$. The (pure) tropical Siegel space $\Hg$ (resp. $\Hgp$) naturally parametrizes marked (pure) tropical p.p. abelian varieties of dimension $g$ (see Proposition \ref{P:Siegel}\eqref{P:Siegel1}). Moreover, there exists  a continuous bijection $\Phi:\Hg\to \Ort$ which restricts to a homeomorphism from $\Hgp$ onto $\O$ (see Proposition \ref{P:Siegel}\eqref{P:Siegel2}). Notice however that the map $\Phi$ is certainly not a homeomorphism because $\Sigma$ is not locally finite at the boundary $\Ort\setminus \O$ (see Lemma \ref{L:locally-finite}).
The topology that we put on $\Hg$ depends on the choice of the admissible decomposition $\Sigma$;
indeed we do not know if, varying $\Sigma$, the tropical Siegel spaces $\Hg$ are homeomorphic among each other or not (see \S\ref{S:open}).

The group $\GL_g(\Z)$ acts naturally on $\Hg$ and on $\Hgp$ by changing the marking and it is easy to check that these actions are admissible (see Lemma \ref{L:action-GL}).
The quotients $\Ag:=\Hg/\GL_g(\Z)$ and $\Agp:=\Hgp/\GL_g(\Z)$ (which exist in the category of ideal stacky fans by what was said before) are moduli spaces for tropical p.p. abelian varieties (resp. pure tropical p.p. abelian varieties) of dimension $g$; they coincide indeed with the moduli spaces first introduced in \cite{BMV} and in \cite{chan}.
Note also that we have natural maps of ideal stacky fans $\Hg\to \Ag$ and $\Hgp\to \Agp$ which correspond to forgetting the marking of the tropical p.p. abelian varieties (resp. the pure tropical p.p. abelian varieties).
In Proposition \ref{P:Siegel-Ag}\eqref{P:Siegel-Ag3}, we prove that the moduli spaces $\Agp$ are homeomorphic to the quotient $\O/\GL_g(\Z)$ for every admissible decomposition $\Sigma$. On the other hand, the topology that we put on $\Ag$ depends on the choice of $\Sigma$ and we do not know if, varying $\Sigma$, the different moduli spaces $\Ag$ are homeomorphic among each other
or not (see \S\ref{S:open}).

\subsection{Tropical period map}

Following again \cite{BMV}, which slightly generalizes the original definition of Mikhalkin-Zharkov \cite{mz}, to any (pure) tropical curve $C=(\Gamma, w, l)$ of genus $g$ we can associate a (pure) tropical p.p. abelian variety of dimension $g$, called the {\em Jacobian}Ê of $C$ and denoted by $J(C)$, which is given by the real torus  $\displaystyle \frac{H_1(\Gamma,\R)\oplus \R^{|w|}}{H_1(\Gamma,\Z)\oplus \Z^{|w|}}$ together with a rational positive semi-definite quadratic form $Q_C$ which is identically zero on $\R^{|w|}$ and on $H_1(\Gamma,\R)$ measures the lengths of the cycles of $\Gamma$ with respect to the length function $l$ on $\Gamma$ (see Definition \ref{D:trop-Jac}). Moreover, a marking $h$ of $C$ naturally induces a marking $\phi_h$ of $J(C)$ (see Definition \ref{D:mark-Jac}).

In Section \ref{S:per-map}, we define the {\bf{\em tropical period map} }
$$\begin{aligned}
{\mathcal P}_g^{tr}: \Tg & \longrightarrow \Ort \\
(C,h)& \mapsto \phi_h^*(Q_C),
\end{aligned}$$
which is shown to be a continuous map in Lemma-Definition \ref{D:periodmap}.

A natural question then arises: can we lift the tropical period map to a map of stacky fans
$$\Pg: \Tg\to \Hg$$
for some admissible decomposition $\Sigma$ of $\Ort$?
The answer is given in Theorem \ref{T:period-S}: such a  map $\Pg$ of stacky fans, called the {\bf {\em $\Sigma$-period map}}, exists  when $\Sigma$ is compatible with the period map
${\mathcal P}_g^{tr}$ in the sense that it sends cells of the stacky fan $\Tg$ into cones of $\Sigma$ (see Definition \ref{D:compat-deco}). For example, it is known that the perfect cone decomposition $\Sigma_P$ and the 2nd Voronoi decomposition $\Sigma_V$ are both compatible with the period map (see Fact \ref{F:compa-P-V}).

It is easily checked that the $\Sigma$-period map (when it exists) is equivariant with respect to the natural group homomorphism
$$\Ab: \Out(F_g)\to \Out(\Z^g)=\Aut(\Z^g)=\GL_g(\Z)$$
induced by the abelianization homomorphism $F_g\to F_g^{\rm ab}=\Z^g$ (see Theorem \ref{T:period-S}\eqref{T:per2}). Therefore, for any $\Sigma$ compatible with the period map, we have the following commutative diagram of stacky fans (which in some sense can be regarded as the main result of this paper):
\begin{equation}\label{E:main-diag}
\xymatrix{
\Tg \ar[r]^{\Pg} \ar[d] &   \Hg \ar[d]\\
\Tg/\Out(F_g)=\Mtrg \ar[r]^{\tg} & \Ag=\Hg/\GL_g(\Z)
}
\end{equation}
where $\tg$ is the map, called the {\bf {\em tropical Torelli map}} (with respect to $\Sigma$), that sends a tropical curve $C$ into its tropical Jacobian $J(C)$ (see Theorem \ref{T:period-S}\eqref{T:per2bis}). For the 2nd Voronoi decomposition $\Sigma_V$, the tropical Torelli map was introduced in \cite{BMV}  and further studied in \cite{chan}.

The restriction of the above diagram \eqref{E:main-diag} to the pure moduli spaces is independent from the chosen admissible decomposition $\Sigma$ and reduces to the following commutative diagram of topological spaces
\begin{equation}\label{E:main-diag2}
\xymatrix{
X_g \ar[r]^{\Pgp} \ar[d] &   \O \ar[d]\\
X_g/\Out(F_g) \ar[r]^{\tgp} & \O/\GL_g(\Z)
}
\end{equation}
where $\Pgp$ (called the {\em pure period map}) is equal to the restriction of the tropical period map $\Pg$  to $X_g\cong \Tgp$
and $\tgp$  (called the {\em pure tropical Torelli map}) is equal to the restriction of the tropical Torelli map $\tg$ (for any $\Sigma$ as above) to $\Mp\cong X_g/\Out(F_g)$; see Theorem  \ref{T:period-S}\eqref{T:per3}. Note that the diagram \eqref{E:main-diag} can be seen as a bordification (or partial compactification) of the diagram \eqref{E:main-diag2}; the vertical arrows are now topological quotient maps, in contrast to the vertical arrows in \eqref{E:main-diag}.

Here is a concrete example that illustrates the effect of applying the maps in (\ref{E:main-diag}) to a particular point in $X_3$.  Let $\G$ be the graph drawn at the top of Figure~\ref{fig:specialize}, with three vertices and edges labeled $a,b,c,d,$ and $e$.  Let $h\colon R_3\ra \G$ be the marking of $\G$ that sends the three loops of $R_3$ to $\ov a b, \ov c b,$ and $ d e \ov d$, respectively, where $\ov e$ denotes the reversal of $e$.  Thus $h$ is a homotopy equivalence.  Suppose $l\colon E(\G)\ra \R_{\ge 0}$  is  a length function on the edges of $\G$.  Then $(\G,l,h)$ is a point in $X_3$, i.e.~a metric graph with a marking.  Applying the period map ${\mathcal P}_g^{tr}$ to $(\G,l,h)$ yields a $3\times 3$ positive definite matrix
$$A = \left(
\begin{matrix}
	l(a)+l(b) & l(b) & 0\\
	l(b) & l(b)+l(c) & 0\\
	0 & 0 & l(e)
\end{matrix}
\right).
$$
The image of $(\G,l,h)$ under the quotient of $X_3$ by $\Out(F_3)$ is simply the pure tropical curve $(\G,l)$ in $M^{tr}_3$.  The period map ${\mathcal P}_g^{tr}$ then descends to the Torelli map $t^{tr}_3$ from $M^{tr}_3$ to $A^{tr}_3$ sending $(\G,l)$ to the $\GL_3(\Z)$-equivalence class of the matrix $A$.  All of these maps will be defined precisely in this paper.

\vspace{0.3cm}

The paper is organized as follows. In Section \ref{s:sf} we introduce the category of  ideal stacky fans and we study quotients of ideal stacky fans by admissible actions. In Section \ref{s:tts} we define the (pure) tropical Teichm\"uller space as a (ideal) stacky fan and study the admissible action of $\Out(F_g)$ on them. In Section \ref{s:mtrg} we realize  the moduli space of (pure) tropical curves as a quotient of the (pure) Teichm\"uller space by $\Out(F_g)$. In Section \ref{S:admi-deco} we study the (pure) tropical Siegel space and identify its quotient by $\GL_g(\Z)$ with the moduli space of (pure) tropical p.p. abelian varieties. In Section \ref{S:periodmap} we study the (pure) tropical period map from the (pure) tropical Teichm\"uller space to the (pure) tropical Siegel space and we identify its quotient with the (pure) tropical Torelli map. Finally, in Section \ref{S:open} we present some open problems that naturally arise from our work.

\vspace{0.2cm}

After the completion of this work, we heard about the Ph.D. thesis \cite{Bak} of Owen Baker at Cornell University, where the tropical period map (called the Jacobian map in loc. cit.)
from Outer space to the space of positive define quadratic  forms is also discussed. More specifically, the author uses the Jacobian map in order to find an invariant deformation retract of Outer Space
for genus $3$, which is then used to compute the first and second cohomology group
of the kernel of the abelianization map $\Ab: \Out(F_3)\to \GL_3(\Z)$ (which can be considered as a tropical analogue of the classical Torelli subgroup of the mapping class group).
We wonder if the results presented here may help in finding an invariant deformation retract of Outer Space in arbitrary genus, generalizing the results of O. Baker.

Moreover, after this paper was submitted for publication, two interesting preprints \cite{Ji} and \cite{ACP} related to the topics discussed in this paper were posted on the arXiv.
In \cite{Ji}, Lizhen Ji uses the tropical period map to construct several $\Out(F_g)$-invariant complete geodesic length metrics on the Outer Space, some of them induced by piecewise smooth
Riemannian metrics.
In \cite{ACP}, D. Abramovich, L. Caporaso and S. Payne describe, among other results,  $\Mtrg$ as the skeleton of the stack $\mathcal{M}_g$ of smooth projective curves of genus $g$.
They view $\Mtrg$ inside the category of generalized cone complexes with integral structure, which is a subcategory of  the category of ideal stacky fans.
Moreover, they construct a tropicalization map
from a suitable subset of the Berkovich analytification ${\mathcal M}_g^{\rm an}$ of ${\mathcal M}_g$ to the tropical moduli space $\Mtrg$, which generalizes the tropicalization map constructed by
the third author in \cite{viviani}.

\vspace{0,3cm}

\noindent {\bf Acknowledgments.} We are extremely grateful to M.~Bestvina and F.~Vallentin for their very helpful correspondences and for generously answering several mathematical questions, and D.~Margalit for supplying many helpful references and for introducing many of the ideas mentioned above to the first author. We thank the referee for her/his useful comments.

M.~Chan was supported by a Graduate Research Fellowship from the National Science Foundation. M. Melo was supported by the FCT project \textit{Espa\c cos de Moduli em Geometria Alg\'ebrica} (PTDC/MAT/ 111332/2009), by the FCT project \textit{Geometria Alg\'ebrica em Portugal} (PTDC/MAT/099275/2008) and by the Funda\c c\~ao Calouste Gulbenkian program ``Est\'imulo \`a investiga\c c\~ao 2010''. F. Viviani is a member of the research center CMUC (University of Coimbra) and he was supported by  the FCT project \textit{Espa\c cos de Moduli em Geometria Alg\'ebrica} (PTDC/ MAT/111332/2009) and by PRIN project
\textit{Geometria delle variet\`a algebriche e dei loro spazi di moduli} (funded by MIUR, cofin 2008).

\section{Ideal stacky fans and stratified quotients}\label{s:sf}

Throughout, we will consider a fixed $\R$-vector space $V\cong\R^\mathcal{I}$ where $|\mathcal{I}|$ is finite or countable. We endow $V$ with the product topology and we fix a lattice $N\cong\Z^\mathcal{I}\subset \R^\mathcal{I}$ in it.

\subsection{Ideal fans and ideal stacky fans}

We start this subsection by recalling the definition of a rational polyhedral fan and introducing the concept of an ideal fan, which is obtained from a rational polyhedral fan by removing some faces.

\begin{defi}\label{D:cone-and-fan}
A  {\em rational polyhedral cone}, or a cone, for short, is a closed, convex set of the form
$$\ov{\sigma}=\{\lambda_1x_1+\cdots\lambda_sx_s~:~\lambda_i\in\R_{\ge0}\}\subseteq V$$
for some finite set $x_1,\ldots,x_s\in N$. A \textit{rational polyhedral fan} in $V$ is a (possibly infinite) collection $\ov{\Sigma}$ of cones  of $V$ satisfying:
\begin{itemize}
	\item if $\ov{\tau}$ is a face of $\ov{\sigma}\in\ov{\Sigma}$ then $\ov{\tau}\in\ov{\Sigma},$ and
	\item the intersection of two cones in $\ov{\Sigma}$ is a face of each.
\end{itemize}
The {\em support} of $\ov{\Sigma}$, denoted $\supp(\ov{\Sigma})$, is the union of the cones in $\ov{\Sigma}$.
\end{defi}

\begin{defi}\label{D:ideal-cone}
An \textit{ideal rational polyhedral cone} $\sigma$ of $V$ (or an \textit{ideal cone} of $V$, for short) is a convex subset of $V$ obtained from an $N$-rational polyhedral cone ${\ov \sigma}$
in $V$ by removing some of its faces. A \textit{face} of an ideal cone $\sigma$ is the intersection of $\sigma$ with a face of $\ov{\sigma}$.  Thus, faces of ideal cones can be empty.  We let $\sigma^0$ denote the relative interior of $\sigma$,
so that $\sigma^0$ is a rational open polyhedral cone.
\end{defi}

Note that faces of $\sigma$ are closed subsets of $\sigma$. The expression ``ideal cones'' is taken
 from ~\cite{cullervogtmann}\footnote{Indeed the word ``ideal" occurs in other contexts with a similar meaning,  e.g. ``ideal triangles'' in hyperbolic geometry
 (see \cite[Prop. 2.4.12]{Thu}) or ``ideal boundaries'' of a Riemann surface (see \cite[Sec. 8.8]{Bea}).}
; when it is clear from the context, we will refer to ideal cones simply as cones.

\begin{defi}\label{D:ideal-fan}
Suppose $\ov{\Sigma} = \{\ov{\sigma_i}\}$ is a rational polyhedral fan in $V$.  Fix some subset of the cones of $\ov{\Sigma}$ whose support is denoted by $Z$.  Then $\supp(\ov{\Sigma})\setminus Z\subseteq V$ can be written as a union of ideal cones in a natural way: if $\ov{\sigma_i}$ was a cone in $\ov{\Sigma}$, then replace it by the ideal cone $\sigma_i$ obtained by deleting the faces of $\ov{\sigma_i}$ lying in $Z$, or remove $\ov{\sigma_i}$ entirely if it lies in $Z$ itself.  We call a collection of ideal cones $\Sigma = \{\sigma_i\}$ obtainable in this way an \textit{ideal fan}.

As with fans, the \textit{support} of an ideal fan is the union of its ideal cones and is denoted $\supp(\Sigma)$.
Note also that
if $\tau$ is a face of $\sigma\in \Sigma$ then $\tau\in \Sigma$, and
the intersection of two ideal cones in $\Sigma$ is a face of each.
\end{defi}

\begin{defi}\label{D:locally-finite}
An ideal fan $\Sigma$ in $V$ is \emph{locally finite} if every $x\in \supp(\Sigma)$ has some open neighborhood that meets only finitely many ideal cones of $\Sigma$.

Equivalently, by shrinking the open neighborhoods, we see that $\Sigma$ is locally finite if and only if every $x\in \supp(\Sigma)$ is in only finitely many ideal cones and has some neighborhood meeting only those cones.
\end{defi}

\begin{figure}\label{non-loc-finite}
\setlength{\unitlength}{1.9pt}
\begin{picture}(113,35.5)(20,148)
\put(47.75,150){\circle{2.693}}
\put(49,150){\line(1,0){35}}
\put(83.75,150){\circle*{2.236}}
\put(83.5,146){\makebox(0,0)[cc]{$x$}}
\put(113,174.75){\makebox(0,0)[cc]{$\ldots$}}
\put(91.75,161.25){\makebox(0,0)[cc]{$\ldots$}}
\put(47.75,150){\line(0,1){26.5}}
\multiput(48,150)(.0336700337,.0883838384){297}{\line(0,1){.0883838384}}
\multiput(48,149.75)(.0337268128,.044688027){593}{\line(0,1){.044688027}}
\multiput(47.75,149.75)(.03848600509,.03371501272){786}{\line(1,0){.03848600509}}
\multiput(47.75,150)(.05166880616,.0336970475){779}{\line(1,0){.05166880616}}
\multiput(47.75,150)(.06450577664,.0336970475){779}{\line(1,0){.06450577664}}
\multiput(47.75,149.75)(.07665394402,.03371501272){786}{\line(1,0){.07665394402}}
\end{picture}
\caption{Example of a non-locally finite ideal fan.  The bottom-left point, colored white, is omitted.  Even though every point of the ideal fan is in finitely many cones, every neighborhood of $x$ meets infinitely many cones.
}
\end{figure}
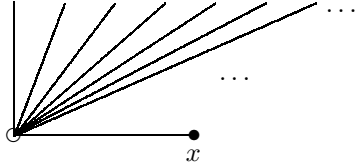

Figure \ref{non-loc-finite} illustrates that the property of being locally finite is stronger than just the requirement that each point is in only finitely many ideal cones.

\begin{lemma}\label{L:locally-finite}
Let $\Sigma$ be an ideal fan in $V$, and let $X_\Sigma = (\coprod \sigma_i)/\sim$ be the space obtained by gluing the ideal cones $\sigma_i\in\Sigma$ on their overlaps.  Then $\Sigma$ is locally finite if and only if the natural  map $X_{\Sigma} \ra \supp(\Sigma)$ is a homeomorphism.
\end{lemma}
\begin{proof}
The map $f\colon \left((\coprod \sigma_i)/\sim\right) \ra\supp(\Sigma)$  is obviously a continuous bijection. Suppose $\Sigma$ is locally finite, let $Z$ be a closed set in $X_\Sigma$, and let $x\in \supp(\Sigma)$ be a limit point of $f(Z)$.  By local finiteness, there is some cone $\sigma$ such that $x$ is a limit point of $f(Z)\cap \sigma$; hence $Z$ is~closed.

Conversely, suppose $\Sigma$ is not locally finite at $x\in \supp(\Sigma)$.  For $j=1,2,\ldots$, choose a cone $\sigma_j \in \Sigma$ whose interior meets the open ball $B_{1/j}(x)$ of radius $1/j$ and such that $\sigma_j$ is not a face of any previously chosen cone.  Such a choice is possible since infinitely many cones meet $B_{1/j}(x)$.  Pick any $x_j \in B_{1/j}(x) \cap\sigma_j^0$ different from $x$.  Then the set $\{x_1,x_2,\ldots\}$ is closed in $X_\Sigma$ since each cone contains only finitely many of its points, but as a subset of $V$, it contains $x$ in its closure.
\end{proof}

\begin{defi}
\label{d:integral}
Let $V=\R^m$ and $V'=\R^n$ with fixed lattices $N\cong\Z^m\subset V$ and $N'\cong\Z^n\subset V' $  and let $X,X'$ be ideal cones in $V$ and $V'$, respectively. We say that a linear map $L:V\ra V'$ is {\em integral} if $L(N)\subseteq N'$.  We say that it is a {\em lattice-preserving} inclusion if it induces an inclusion $X\hookrightarrow X'$ identifying $X$ with a face of $X'$, and if
$$L(N\cap {\rm span}(X)) = N'\cap L({\rm span}(X)).$$
\end{defi}

We now introduce ideal stacky fans, which are a slight strengthening of stacky fans in the sense of  \cite[Def. 3.2]{chan} (relaxing the
definition in~\cite[Def. 2.1.1]{BMV}).

\begin{defi}
\label{d:sf}
Let $\{V_i\}$ is a collection of finite-dimensional real vector spaces, each with a fixed associated lattice,
and let $\{\sigma_i\subseteq V_i\}$ be a collection of ideal cones, one in each.  Let $\{L_\al\colon V_i\ra V_j\}$ be a collection of lattice-preserving linear maps inducing an identification $\sigma_i\hookrightarrow \sigma_j$ of $\sigma_i$ with a face of $\sigma_j$.  Here, we allow $i=j$.

Let $X=(\coprod \sigma_i)/\sim$, where $\sim$ is the equivalence relation generated by identifying each $x$ with $L_\al(x)$ for all linear maps $L_\al$.  If $X=\coprod \left( \sigma^0_i/\!\sim\right)$ as sets, then we say that $X$ is an {\em ideal stacky fan} with cells $\{\sigma^0_i/\!\sim\}$.  We call a map $L_\alpha$ an {\em inclusion of faces}.  We say that $\sigma_i$ is a {\em stacky face} of $\sigma_j$ if $i=j$ or there is a sequence of inclusions of faces from $\sigma_i$ to $\sigma_j$.
\end{defi}

\begin{remark}\label{r:finitely-many-maps}
Note that for fixed $\sigma$ and $\sigma'$ in the definition above, there can be only finitely many distinct maps $L_\alpha\colon\sigma\!\ra\sigma'$.  For since $L_\al$ is lattice-preserving, it must take the first lattice point on a ray of the closure of $\sigma$ to a first lattice point on a ray of the closure of $\sigma'$, so there are only finitely many choices.
\end{remark}

\begin{remark}\label{R:old-def}
Definition~\ref{d:sf} allows for ideal cones instead of closed cones, and it allows for infinitely many cones as well.  However, a straightforward argument, essentially the one in \cite[Theorem~3.4]{chan} in the case of tropical moduli space, shows that an ideal stacky fan on finitely many closed cones is a stacky fan in the restricted sense of~\cite[Def.~3.2]{chan}.
\end{remark}

\begin{remark}\label{r:asso-fan}
If $\Sigma=\{{\sigma}_i\}$ is an ideal fan in a finite-dimensional vector space $V$, then the space $X_\Sigma=(\coprod \sigma_i)/\!\sim$, where $\sim$ is generated by inclusions of faces, is trivially an ideal stacky fan with cells $\{\sigma_i^0\}$.  However, $X_\Sigma$ has the same topology as $\supp(\Sigma)\subseteq V$ only in the case that $\Sigma$ is locally finite; this is precisely Lemma~\ref{L:locally-finite}.
\end{remark}

\begin{remark}\label{r:infinite}
Suppose instead that $V$ is infinite-dimensional, $\Sigma=\{{\sigma}_i\}$ is an ideal fan in $V$, and each cone of $\Sigma$ is simplicial, i.e.~is a cone over a simplex, possibly minus some faces.
Pick a point $e_\rho$ on each ray $\rho$ of the associated (closed) fan $\ov{\Sigma}$, to be regarded as the ``first lattice point'' of that ray.  Then a $d$-dimensional cone $\sigma\in\Sigma$ can
be identified with an orthant $\sigma' = \R_{\ge 0}^d \subset \R^d$, possibly minus some faces, by sending each point $e_\rho$ to one of the standard basis vectors in $\R^d$.

Then we associate a space $X_\Sigma$ to $\Sigma$ by gluing the cones $\sigma'$ along their faces in accordance with the fan structure of $\Sigma$; this space $X_\Sigma$ is again trivially a
stacky fan.  Again by Lemma~\ref{L:locally-finite}, it is homeomorphic to $\supp(\Sigma)$ if and only if $\Sigma$ is locally~finite.
\end{remark}

Ideal stacky fans enjoy remarkable topological properties.

\begin{prop}\label{P:CW}
Any ideal stacky fan is homeomorphic to a cone over a CW-complex minus a subcomplex.
\end{prop}
\begin{proof}
Let $X=(\coprod \sigma_i)/\!\sim$ be an ideal stacky fan.
Assume first that each $\sigma_i$ is a closed cone. Take the barycentric subdivision of each $\sigma_i$; let $\{\sigma_i^1,\ldots,\sigma_i^{l_i}\}$ be the cones in the resulting complex.  Let $C = \{\sigma^k_i\}_{\sigma_i\in\Sigma, 1\le k\le l_i}$ be the set of all of the barycentric pieces of the cones.
Now, each $L_\alpha$ maps each cone in $C$ homeomorphically to another cone in $C$.  Consider the equivalence relation on the set $C$ generated by the maps $L_\al$ in this way, and let $\mathcal{J}\subseteq C$ be a choice of representatives.   Then
$$X\cong \displaystyle \coprod_{\sigma^k_i\in \mathcal{J}} \sigma^k_i /\sim',$$
where $\sim'$ is generated by composing identifications via $L_\al$ and identifications of faces within a barycentric subdivision.  Note in particular that any such composite map taking $\sigma^k_i$ to itself must be the identity, since the vertices of $\sigma^k_i$ correspond to a faces of $\sigma_i$ of distinct dimensions and each map $L_\al$ preserves this correspondence.  It follows that the space $\coprod_{\mathcal{J}} \sigma^k_i /\sim'$ is a cone over a CW-complex.

For a general ideal stacky fan, we may remove some of the faces of the cones $\sigma_i$ and repeat the above construction; hence we end up with a cone over a CW-complex minus a subcomplex.
\end{proof}

\begin{cor}\label{c:all-hausdorff}
Any ideal stacky fan is
\begin{itemize}
\item normal (hence Hausdorff);
\item paracompact;
\item locally contractible (hence locally path connected and locally connected).
\end{itemize}
\end{cor}
\begin{proof}
Any CW-complex $Y$ is normal (see \cite[Proposition A.3]{hatcher}), paracompact (see \cite[Ex. 4, Sec. 4G, p. 460]{hatcher}) and locally contractible (see \cite[Prop. A.4]{hatcher}).
It is easily checked that these properties are preserved if we remove a subcomplex $A$ from $Y$ and then take the cone over $Y\setminus A$.
Hence the result follows from Propostion \ref{P:CW}.
\end{proof}

Thus, all of the spaces we shall consider in this paper are normal.  In particular, this proof unifies and generalizes the results \cite[Theorem 5.2]{caporaso-linkage} and \cite[Theorem 4.13]{chan} that the spaces $\Mtrg$ and $\Atrg$, whose definitions we will soon recall, are Hausdorff.

Ideal stacky fans that satisfy a mild finiteness property enjoy further topological properties.

\begin{prop}\label{p:sf-hausdorff}
Let $X=(\coprod \sigma_i)/\!\sim$ be an ideal stacky fan, and assume that each ideal cone $\sigma_i$ is a stacky face of only finitely many other ideal cones.
Then $X$ is:
\begin{itemize}
\item locally compact;
\item metrizable (hence first countable).
\end{itemize}
If, moreover, $X$ has only countably many connected components, then $X$ is also
\begin{itemize}
\item second countable (hence separable and Lindel\"of).
\end{itemize}
\end{prop}

\begin{proof}
Let $X=(\coprod \sigma_i)/\!\sim$ be an ideal stacky fan satisfying the hypothesis.
Since each $\sigma_i$ is a stacky face of only finitely many other cones, it follows from Remark~\ref{r:finitely-many-maps} that the surjective continuous map
$(\coprod \sigma_i) \ra (\coprod \sigma_i)/\!\sim$ is closed  with finite fibers.  In particular, it is  a \emph{perfect} map, i.e.~a surjective continuous map which is closed and has compact
fibers (see \cite[Ex. 26.12]{munkres}).  Since the topological space $\coprod \sigma_i$ is locally compact (because each ideal cone $\sigma_i$ is),
and it is well-known that  a perfect map preserves local compactness (see e.g. \cite[Ex.~31.7]{munkres}), we conclude that $X$ is locally compact.

Observe next that, since $X$ is locally path connected, $X$ is the disjoint union of its connected components (with the topology of the disjoint union).
In order to prove that $X$ is metrizable, it is sufficient to prove that each of its connected components is metrizable.
Clearly, each connected component $C$ of $X$ inherits from $X$ the property of being locally compact, regular and paracompact (see also Corollary \ref{c:all-hausdorff}).
 Hence $C$ is second countable by
\cite[Ex.~41.10]{munkres}. Since $C$ is regular and second countable, the Urysohn metrization theorem (see \cite[Thm. 34.1]{munkres})
implies that $C$ is metrizable.
Finally, if $X$ has countably many connected components, then since each of these connected components is second countable, it follows that $X$
 itself is second countable.

\end{proof}

\begin{remark}
The hypothesis  that each ideal cone $\sigma_i$ is a stacky face of only finitely many other ideal cones
 cannot be removed from Proposition \ref{p:sf-hausdorff}, as the following example shows.

Consider the stacky fan $X$ whose cones are $\sigma_0:=\{0\}\subseteq \R^0$ and $\sigma_n:=\R_{\geq 0}\subset \R^1$ for
every $n\geq 1$, and whose inclusion of faces are $L_n:\sigma_0\hookrightarrow \sigma_n$ identifying
$\sigma_0$ with the origin of $\sigma_n$, for every $n\geq 1$.
It is easily checked that $X$ is not first countable at the point corresponding to $\sigma_0$ (hence it is not metrizable)
and it is not locally compact.
\end{remark}

There is a natural notion of morphism of ideal stacky fans that generalizes the definition  of morphism of stacky fans (see \cite[Def. 2.1.2]{BMV}).

\begin{defi}\label{d:sfm}
Let $X$ and $Y$ be ideal stacky fans with cells $\{\sigma^0_i/\!\sim\}$ and $\{\tau^0_i/\!\sim\}$, respectively.  A continuous map $f\colon X\ra Y$ is a {\em morphism of ideal stacky fans} if for all $\sigma_i$, there exists a $\tau_j$ and an 
integral--linear map $L:\sigma_i\ra\tau_j$ such that the following diagram commutes.
\[ \xymatrix{
\sigma_i \ar[r]^L \ar[d]  &\tau_j \ar[d]\\
X \ar[r]^f &Y } \]

\end{defi}
\subsection{Admissible decompositions and stratified quotients}\label{ss:strat}

The aim of this subsection is to introduce admissible decompositions and their corresponding stratified quotients.

\begin{defi}\label{strat-defi}
Let $V$ be a finite-dimensional real vector space, let $X\subseteq V$, and let $G$ be a group acting on $X$. An ideal fan $\Sigma$ with support $X$ is called
an \textit{ideal G-admissible decomposition} of $X$ if:
\begin{enumerate}[(i)]
\item the action of $G$ permutes the cones on $\Sigma$, that is, if ${ \sigma}\in\Sigma$ and $g\in G$, then $\sigma \cdot g\in\Sigma$;
\item the action of $g\in G$ on any $\sigma\in\Sigma$ is given by a lattice-preserving linear map $V\ra V$ taking $\sigma$ to $\sigma\cdot g$.
\end{enumerate}
\end{defi}

Given an ideal $G$-admissible decomposition $\Sigma$ of $X$, we now define a kind of quotient space, called the stratified quotient, obtained by gluing representative cones together. We will see that it is the same as the topological quotient $X/G$ in the case that $\Sigma$ is locally finite, but is better behaved than $X/G$ if not
(see Corollary \ref{p:homeo} and Remark \ref{R:nonlocfin-dec}).

\begin{defi}\label{d:sq}
Let $V$ be finite-dimensional, let $\Sigma$ be an ideal $G$-admissible decomposition of $X\subseteq V$, and let $\{\sigma_i\}$ be a system of representatives for the $G$-orbits of the cones in $\Sigma$. Given two representatives $\sigma_i$ and $\sigma_j$ and an element $g\in G$ such that $\sigma_i\cdot g$ is a face of $\sigma_j$, let $L_{i,j,g}: \sigma_i\hookrightarrow\sigma_j$ be the corresponding lattice-preserving linear map.  (Here, we allow $i=j$.)  Then the \textit{stratified quotient} of $X$ with respect to $\Sigma$ is
$$X/\!/_{\Sigma} G:=\left(\coprod\sigma_i\right)/_{\sim}$$
where $\sim$ is the equivalence generated by the maps $L_{i,j,g}$.  We emphasize that $X/\!/_{\Sigma} G$
is a topological space with respect to the quotient topology.
\end{defi}

With just a little more work, we can define $G$-admissible actions and stratified quotients for any ideal stacky fan.  We will need these more general definitions to take quotients of Outer Space and its simplicial closure in Section~\ref{s:tts}.

\begin{defi}\label{d:admissible-action}
Let $X=(\coprod{\sigma_i})/\!\sim$ be a stacky fan with cells $\{\sigma^0_i/\!\sim\}$. An action of a group $G$ on $X$ is {\em admissible} if for each $g\in G$ and $\sigma_i$, there exists some $\sigma_j$ (necessarily unique) and a lattice-preserving map $L_{g,i}\colon\sigma_i\stackrel{\cong}{\longrightarrow}\sigma_j$ such that the following diagram of sets commutes.
\[ \xymatrix{
\sigma_i \ar[r]^{L_{g,i}} \ar[d]  &\sigma_j \ar[d]\\
X \ar[r]^g &X } \]
We say that $\sigma_i$ and $\sigma_j$ are in the same $G$-orbit if so.
\end{defi}

\begin{defi}\label{d:sq2}
Let $X=(\coprod{\sigma_i})/\!\sim$ be a stacky fan with cells $\{\sigma^0_i/\!\sim\}$ and inclusions of faces $\{L_\al\}$, and suppose we have an admissible $G$-action on $X$.  Let $\{\sigma_j\}_{j\in\mathcal{J}}$ be a set of $G$-orbit representatives.  Then define the \textit{stratified quotient} to be
$$X/\!/\,G= (\coprod_{j\in\mathcal{J}}{\sigma_j})/\!\sim',$$
where the quotient is over all identifications given by composite maps of the form $\sigma_i \stackrel{L_{g,i}}{\longrightarrow}\sigma_i\cdot g\stackrel{L_\alpha}{\longrightarrow}\sigma_j$
for all $g,i,$ and $j$.
\end{defi}

\begin{remark}\label{r:spec-cases}
Definitions~\ref{strat-defi} and~\ref{d:sq} are just special cases of Definitions~\ref{d:admissible-action} and~\ref{d:sq2}, namely, in the case that the ideal stacky fan is obtained from an ideal polyhedral fan.
\end{remark}

\begin{prop}\label{p:isstackyfan}
Let $X=(\coprod{\sigma_i})/\!\sim$ be an ideal stacky fan with cells $\{\sigma^0_i/\!\sim\}$ and inclusions of faces $\{L_\al\}$, and suppose we have an admissible $G$-action on $X$.  Let $\{\sigma_j\}_{j\in\mathcal{J}}$ be a set of $G$-orbit representatives.
Then the stratified quotient $X/\!/\, G$ is an ideal stacky fan with cells $\{\sigma^0_j/\!\sim'\}_{j\in\mathcal{J}}$, and the map $X\ra X/\!/\,G$ is a morphism of ideal stacky fans.
\end{prop}

\begin{proof}
The only part of the first claim that needs checking is that the obvious map from $\coprod_{j\in\mathcal{J}} (\sigma^0_j/\!\sim')$ to $X$ is a bijection.  It is a surjection since $\{\sigma_j\}$ is a set of representatives; on the other hand, if the images of $\sigma^0_j$ and $\sigma^0_{j'}$ overlap, then some $g\in G$ takes $\sigma_j$ to $\sigma_{j'}$, so $j=j'$.

To show that $X\ra X/\!/\,G$ is a morphism, consider the following diagram of topological spaces.
\[ \xymatrix{
\coprod \sigma_i \ar[r] \ar[d]^{/\{L_\al\!\}}  &\displaystyle\coprod_{j\in\mathcal{J}}{\sigma_j} \ar[d]^{/\{L_\al \circ L_{g,i}\!\}}\\
X \ar@{.>}[r] &X/\!/\,G  }\]
The top map sends each cone $\sigma_i$ to its representative via a lattice-preserving linear map $L_{g,i}$.  The bottom map exists (and is continuous), and makes the square commute, because the left arrow is a quotient map.
The maps $L_{g,i}$ comprising the top arrow of the commmutative square also give precisely the condition that $X\ra X/\!/\,G$ is a morphism of ideal stacky fans.
\end{proof}

A priori, the definition of a stratified quotient depends on the choice of the representatives for the $G$-orbits on $\Sigma$. However, we will prove 
next that this is not the case
and therefore that the above definition is well-posed, not just at the level of topological spaces (see Proposition~\ref{p:strat-is-global}), but as ideal stacky fans.

\begin{prop}\label{d:indip-choice}
  The construction of $X/\!/\,G$ in Definition \ref{d:sq2} does not depend on our choice of
  representatives $\{\sigma_j\}_{j\in\mathcal{J}}$. More precisely, suppose that $\{\sigma'_j\}_{j\in\mathcal{J}}$ is another choice of representatives such that
  ${ \sigma}_j'$ and ${ \sigma}_j$ are $G$-equivalent for each $j$. Let $X/\!/\,G$ and $X\wt{/\!/\,} G$ denote the respective corresponding ideal stacky fans. Then there is an isomorphism
  of stacky fans between $\, X/\!/\, G \,$ and $X\wt{/\!/\,}G$.
\end{prop}

\begin{proof}
  For each $j$, choose $g_j \in G$ with
  \[ \sigma_{j}\cdot g_j = { \sigma_{j}'}. \]
  Then we obtain a map
  \[\coprod_{j\in\mathcal{J}}(\sigma_j/\!\sim)\xrightarrow{(g_j)_{j\in\mathcal{J}}}\coprod_{j\in\mathcal{J}}(\sigma_j/\!\sim)\]
    descending to a map
  \[X/\!/\, G \longrightarrow X\wt{/\!/\,} G,\]
  and this map is an isomorphism of stacky fans, as evidenced by the inverse map $X\wt{/\!/\,} G \rightarrow X/\!/\, G$ constructed from the elements $\{g_j^{-1}\}_{j\in\mathcal{J}}$.
\end{proof}

From Corollary \ref{c:all-hausdorff}, we get the following result.

\begin{cor}\label{c:hausdorff}
A stratified quotient is normal (hence Hausdorff).
\end{cor}

In fact, stratified quotients are global quotients. (By the global quotient $X/G$ we just mean the set of $G$-orbits of $X$, endowed with the quotient topology.)

\begin{prop}\label{p:strat-is-global}
Let $X=(\coprod{\sigma_i})/\!\sim$ be an ideal stacky fan, and suppose we have an admissible $G$-action on $X$.  Then we have a homeomorphism
$$X/G \cong X/\!/\,G.$$
\end{prop}

\begin{proof} As usual, let $L_\al$ denote the face inclusions of $X$, let $L_{g,i}$ be the linear map induced on a cone $\sigma_i$ by $g\in G$ as in Definition~\ref{d:admissible-action}, and let $\mathcal{J}$ denote a subcollection of cones that form a set of orbit representatives in the $G$-action on $X$.

We have a surjective map $\pi\colon \coprod{\sigma_i}\twoheadrightarrow \coprod_{j\in\mathcal{J}}{\sigma_j}$ sending each cone to its representative, and a section $i\colon \coprod_{j\in\mathcal{J}}{\sigma_j}\hookrightarrow\coprod{\sigma_i}.$
The space $X/G$ is obtained from $\coprod{\sigma_i}$ by quotienting by the face inclusions $L_\al$ to obtain $X$, and then taking the quotient of $X$ by the action of $G$.
The space $X/\!/G$ is obtained from $\coprod_{j\in\mathcal{J}}{\sigma_j}$ by quotienting by maps of the form $L_\al\circ L_{g,i}$.  These maps are related by the diagram below, and give rise to the maps $f$ and $g$, depicted by dotted arrows, which make both the forward and the backward squares commute.
$$\xymatrix{\coprod{\sigma_i} \ar[dd]^{/L_\al,G}  \ar[rr]<+.5ex>^{\pi} &&   **[r]{\displaystyle \coprod_{j\in\mathcal{J}}{\sigma_j} \ar[ll]<+.5ex>^i} \ar[dd]^{/L_\al\! \circ L_{g,i}}\\
&&\\
X/G \ar@{.>}[rr]<+.5ex>^f &&X/\!/\,G \ar@{.>}[ll]<+.5ex>^h }$$
Finally, from the fact that $\pi\circ i = {\rm id}$ and $i\circ\pi$ sends a point to something in its $G$-orbit, we conclude that $h\circ f = {\rm id}$ and $f\circ h = {\rm id}$.
\end{proof}

\begin{remark}\label{quotient-object}
It follows from Proposition~\ref{p:strat-is-global} that the stratified quotient $X/\!/G$ really is a quotient object with respect to $G$-equivalence in the category of ideal stacky fans, in the sense that any morphism $X\ra Y$ that respects $G$-equivalence factors uniquely as a composite map $X\ra X/\!/G \ra Y$ of ideal stacky fans.
\end{remark}

Specializing Proposition \ref{p:strat-is-global} to the case of admissible $G$-actions on the ideal stacky fans coming from a locally finite $G$-admissible decomposition of $X\subseteq V$, we get the following

\begin{cor}\label{p:homeo}
Suppose that $\Sigma$ is a locally finite $G$-admissible decomposition of $X\subseteq V$, and let $X_\Sigma$ be the stacky fan structure on $X$ (see Lemma~\ref{L:locally-finite}).
Then we have a homeomorphism
$$X/G\cong X_\Sigma/\!/G.$$
\end{cor}
\noindent In other words, for locally finite fans that are $G$-admissible decompositions, stratified quotients are global quotients.
\begin{proof}
We have homeomorphisms $X\cong X_\Sigma$ and $X_\Sigma/G \cong X_\Sigma/\!/G$ by Lemma~\ref{L:locally-finite} and Proposition~\ref{p:strat-is-global}.
\end{proof}

\begin{remark}\label{R:nonlocfin-dec}
 For an arbitrary (not necessarily locally finite) $G$-admissible decomposition $\Sigma$ of $X\subseteq V$,
 the continuous bijection $X_{\Sigma}\to X$ of Lemma~\ref{L:locally-finite} induces a continuous  bijection
$$X /\!/_{\Sigma}G\cong X_{\Sigma}/\!/G\cong X_{\Sigma}/G\longrightarrow X/G,$$
where the homeomorphisms on the left hand side follow from Remark~\ref{r:spec-cases} and Proposition \ref{p:strat-is-global}. However, it is easy to construct examples of non-locally finite decompositions $\Sigma$ where the above continuous bijection is not a homeomorphism, since the space on the left hand side is always Hausdorff (see Corollary \ref{c:hausdorff}) while the space on the right hand side could very well be non-Hausdorff.
\end{remark}

\section{Tropical Teichm\"uller space}\label{s:tts}

The aim of the present section is to introduce two spaces: the \emph{pure tropical Teich\-m\"uller space}, which we will denote by $\Tgp$, together with a closure of it, which we will call \emph{tropical Teichm\"uller space} and denote by $\Tg$. These spaces parametrize stable metric graphs (respectively stable metric weighted graphs) of genus $g$ together with a fixed isomorphism of their fundamental group with the free group on $g$ letters $F_g$.

The idea of considering spaces parametrizing stable metric graphs together with such a marking is well-known in geometric group theory and is due to Culler and Vogtmann, who introduced in \cite{cullervogtmann} the space $X_g$, now called Outer space, parametrizing such objects. There, the authors endow Outer space with a topology by embedding it in an infinite-dimensional vector space; they also consider its closure $\ov X_g$ in this space.
The outer automorphism group $\Out(F_g)$ of the free group $F_g$ acts properly discontinuously (hence with finite stabilizers) on $X_g$ by changing the marking.  The approach of deducing cohomological information on the group $\Out(F_g)$ via its action on $X_g$ has been extremely fruitful; we refer the reader to \cite{vogtmann} and \cite{vogICM} for a survey of the known results.

The objects parametrized by our pure tropical Teichm\"uller space are essentially the same objects parametrized by Outer space, the only difference being one of convention: we do not
normalize edge lengths as Culler and Vogtmann do.
 We do, however, endow $\Tgp$ and $\Tg$
 with a topology in a different way: our strategy is to consider cells consisting of marked metric graphs (resp. marked weighted metric graphs) of the same topological type and glue them together in a similar way to what was done in \cite{BMV} in order to construct the moduli space of tropical curves.  The resulting spaces are manifestly ideal stacky fans.
In fact $\Tgp$ is homeomorphic to (non-normalized) Outer space  (see Corollary~\ref{T:tgp-is-xg}); in other words, the induced topology on Outer space coincides with the simplicial topology on it.  On the other hand, the ideal stacky fan $\Tg$ obtained from the simplicial structure on $\ov X_g$ has a topology that is actually finer than the subspace topology on $\ov X_g$; see Lemma~\ref{L:locally-finite}.

We have rechristened $X_g$ in this paper simply to emphasize that our reasons and context for studying it are quite different from the usual ones; and to emphasize that $\Tgp$ is an object in the category of ideal stacky fans.
We next describe the construction of $\Tgp$ and of $X_g$, following \cite[Part I]{vogtmann}.

\subsection{Pure tropical Teichm\"uller space}\label{ss:tgp}

Recall that given a graph $\Gamma$ (possibly with loops or multiple edges), with vertex set $V(\Gamma)$ and edge set $E(\Gamma)$,
the  \textit{genus} of $\Gamma$ is  $g(\Gamma):=|E(\Gamma)|-|V(\Gamma)|+1$; it is the dimension of the vector space generated by the cycles of $\Gamma$.
The \textit{valence} of a vertex $v$, $\val(v)$, is defined as the number of edges incident to $v$, with the usual convention that a loop around a vertex $v$ counts twice.
We say that a graph $\Gamma$ is {\em stable} if any vertex of $\Gamma$ has valence at least two.


\begin{defi}\label{metric-graph}
A {\em metric graph} is a graph $\Gamma$ together with a length function $l:E(\Gamma)\to \R_{>0}$.
The {\em volume} of a metric graph $(\Gamma,l)$ is the sum of the lengths of the edges of $\Gamma$.  We can regard $(\Gamma,l)$ as a metric topological space.

A {\em pure tropical curve} $C$ of genus $g$ is a metric graph $(\Gamma,l)$  such that $\Gamma$ is a stable graph of genus $g(\Gamma)=g$.

\end{defi}

\noindent Pure tropical curves are special tropical curves, as we will see in Subsection \ref{ss:tg}. The term pure was introduced by L. Caporaso in  \cite{capsurvey}.

Now fix a graph $R_g$ with one vertex $v$ and $g$ edges (a {\em rose with $g$-petals}) and identify the free group $F_g=F\langle x_1,\dots, x_g \rangle$ with $\pi_1(R_g, v)$ in such a way that each generator $x_i$ corresponds to a single oriented edge of $R_g$. Under this identification, reduced words in $F_g$ correspond to reduced edge-path loops starting at the vertex $v$  of $R_g$, and  therefore  we will make no distinction between them.

\begin{defi}\label{D:marked-graph}
\noindent
\begin{enumerate}[(i)]
\item Let $\Gamma$ be a graph of genus $g$. A {\em marking} on $\Gamma$ is a homotopy equivalence $h:R_g\to\Gamma$.  Here, $\Gamma$ is viewed as a 1-complex, with a free fundamental group of rank $g$. We say that the pair $(\Gamma,h)$ is a {\em marked graph} of genus $g$.
We regard $(\Gamma,h)$ and $(\Gamma',h')$ as equivalent if there is an isomorphism of 1-complexes $\gamma:\Gamma\to\Gamma'$ with $h\circ \gamma$ homotopic to $h'$.

We are always interested in marked graphs only up to the above equivalence, but for simplicity will just say ``marked graph'' instead of ``equivalence class of marked graphs'' throughout.
\item A {\em marked metric graph}  $(\Gamma,l,h)$ consists of a metric graph $(\Gamma,l)$ together with a marking $h:R_g\to \Gamma$ of the underlying graph.
We say that two marked metric graphs $(\Gamma,l,h)$ and $(\Gamma',l',h')$ are equivalent if there is an isometry $\gamma:(\Gamma,l)\to (\Gamma', l')$ with $h\circ \gamma$ homotopic to $h'$.  Again, we will always consider marked metric graphs only up to equivalence in this section, and for simplicity we will write ``marked metric graph'' instead of ``equivalence class of marked metric graphs.''
\item A {\em marked pure tropical curve} $(C,h)=(\G,l,h)$ of genus $g$ is a pure tropical curve $C=(\Gamma,l)$ of genus $g$
together with a marking $h:R_g\to \Gamma$ of the underlying graph $\G$, up to equivalence.
The stable marked graph $(\Gamma,h)$ is called the combinatorial type of the marked pure tropical curve.
\end{enumerate}
\end{defi}

Our next goal is to define the pure tropical Teichm\"uller space $\Tgp$, a space that will parametrize marked pure tropical curves of genus $g$.
We will show that it is an ideal stacky fan in Proposition~\ref{p:tgp-is-isf}.  We start by defining its cells.

\begin{defi}\label{D:simplicial-cones}
Given a stable marked graph $(\Gamma, h)$ of genus $g$, fix a numbering on its set of edges $E(\Gamma)$. Let $\C_{(\Gamma,h)}^0:=\mathbb R^{|E(\Gamma)|}_{>0}$ be the open simplicial cone  of  $\mathbb R^{|E(\Gamma)|}$, and write $\ov{C_{(\G,h)}} = \mathbb R^{|E(\Gamma)|}_{\ge0}$ for its closure.
\end{defi}
\noindent So, for example, graphs with $k$ edges correspond to cones of dimension $k$, and an 
Euler characteristic argument shows that all cones have dimension at most $3g-3$.

\begin{remark}\label{r:no-stackiness}
The points of $\C_{(\Gamma,h)}^0$ are in bijection with (equivalence classes of)
 marked metric graphs $(\Gamma,l,h)$ for some length function $l$, or in other words with  pure tropical curves whose underlying combinatorial type is $(\Gamma, h)$.  Indeed, a point in $\mathbb R^{|E(\Gamma)|}_{>0}$ determines a length function $l\colon E(\G)\ra\R_{>0}$.  If $l'=l\circ p$ is another length function, where $p\colon E(\G)\ra E(\G)$ is a permutation induced by a nontrivial isometry $\psi$, then the marking $\psi\circ h$ must be different from $h$; for since $\Gamma$ is not homeomorphic to a circle, $\psi$ cannot fix every loop of $\Gamma$.
\end{remark}

The pure tropical Teichm\"uller space will be obtained by gluing  certain partial closures  $\C_{(\Gamma,h)}$ of $\C^0_{(\Gamma,h)}$ along ideal faces
corresponding to specializations of $(\Gamma,h)$, as we are now going to define.

\subsubsection{Specializations of marked graphs}\label{specializations}

\begin{defi}\label{D:spec-pure}
Let $(\Gamma,h)$ and $(\Gamma',h')$ be two marked graphs of the same genus $g$.
We say that $(\Gamma,h)$ {\em specializes} to $(\Gamma',h')$, and we write $(\Gamma,h)\rightsquigarrow (\Gamma',h')$, if
there is a surjective morphism of  graphs $\pi:\Gamma\to\Gamma'$ induced by contracting an acyclic subgraph of $\Gamma$ making the following diagram homotopy commutative.
\begin{equation*}
\xymatrix{
{R_g} \ar[r]^h \ar[dr]_{h'}& {\Gamma} \ar[d]_{\pi}\\
&\Gamma'
}
\end{equation*}
Note that if $(\Gamma,h)\rightsquigarrow (\Gamma',h')$ and $\Gamma$ is stable then also $\Gamma'$ is stable.
\end{defi}

\begin{defi}
Let $(\Gamma, h)$ be a stable marked graph of genus $g$.  Given a subset $S\subseteq E(\G)$, let $F_S$ denote the face of $\ov{\C_{(\Gamma,h)}}$ corresponding to those length functions $l\colon E(\G)\ra \R_{\ge 0}$ that are zero on all edges in $S$.  Then define $\C_{(\Gamma,h)}$ to be the ideal cone obtained from $\ov{\C_{(\Gamma,h)}}$  by removing those faces corresponding to sets $S\subseteq E(\G)$ containing a cycle.
\end{defi}

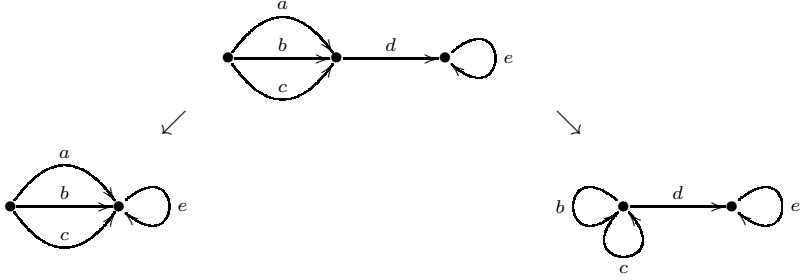
\begin{figure}%
\begin{center}
\begin{tabular}{rcl}
&\xymatrix@=3pc{
*{\bullet} \ar@{->}^{a}@/^1.3pc/[r] \ar@{->}^{c}@/_1.3pc/[r] \ar@{->}^b[r]  & *{\bullet}
\ar@{->}^d[r] & *{\bullet} \ar@{->}^{e}@(ur,dr)}
&\\
\xymatrix{\swarrow}& & \xymatrix{\searrow}
\\
\xymatrix@=3pc{
*{\bullet} \ar@{->}^{c}@/_1.3pc/[r] \ar@{->}^b[r]  & *{\bullet}
\ar@{<-}_{a}@/_1.3pc/[l] \ar@{->}^{e}@(ur,dr)
}
& &
\xymatrix@=3pc{
 *{\bullet}  \ar@{->}_{b}@(ul,dl)  \ar@{->}_{c}@(dl,dr)
\ar@{->}^d[r] & *{\bullet} \ar@{->}^{e}@(ur,dr)
}
\end{tabular}
\end{center}
\caption{Two specializations of a graph $\G$ of genus 3. Denote by $\ov e$ the reversal of the directed edge $e$.  Suppose the marking of $\G$ sends the three loops of $R_3$ to $\ov a b, \ov c b,$ and $ d e \ov d$, respectively.  Then the marking specializes, on the left, to $(\ov a b, \ov c b, e)$ and, on the right, to $(b, \ov c b, d e \ov d)$.}
\label{fig:specialize}%
\end{figure}

Thus, the nonempty faces of $C_{(\G,h)}$ correspond to (equivalence classes of) specializations of $(\G,h)$.  Indeed, a specialization $(\Gamma,h)\rightsquigarrow(\Gamma',h')$ yields an obvious inclusion $\iota\colon E(\Gamma')$ $ \ra E(\Gamma)$ that induces, in turn, a lattice-preserving linear map $L_\iota\colon \C_{(\Gamma',h')}\hookrightarrow\C_{(\Gamma,h)}$ sending $\C_{(\Gamma',h')}$ to a face of $\C_{(\Gamma,h)}$.  Using these linear maps, given a stable marked graph $(\Gamma,h)$, we have a natural identification of sets
\begin{equation}\label{celtrop-dec}
\C_{(\Gamma,h)}=\coprod_{(\Gamma,h)\rightsquigarrow(\Gamma',h')}\C^0_{(\Gamma',h')},
\end{equation}
where the union runs over all equivalence classes of stable marked graphs of genus $g$ which are obtained as specializations of $(\Gamma,h)$.
Summarizing, we have:

\begin{lemma}\label{L:specialization-ideal}
Let $(\Gamma, h)$ be a stable marked graph of genus $g$. Then the faces of ${\C_{(\Gamma,h)}}$ are in bijective correspondence with ideal cones $\C_{(\Gamma',h')}$, where  $(\Gamma',h')$ is a stable marked weighted graph of genus $g$ such that $(\Gamma, h) \rightsquigarrow (\G',h')$.  The identification is given by a lattice-preserving linear map $L_\iota\colon \C_{(\Gamma',h')}\hookrightarrow\C_{(\Gamma,h)}$ sending $\C_{(\Gamma',h')}$ to a face of~$\C_{(\Gamma,h)}$.
\end{lemma}

\subsubsection{The topology underlying pure tropical Teichm\"uller space.}

We are now ready to define the pure tropical Teichm\"uller space $\Tgp$.

\begin{defi}\label{Tgp}
The pure tropical Teichm\"uller space of genus $g$ is the topological space (with respect to the quotient
topology)
$$\Tgp:=\left(\coprod {\C_{(\Gamma,h)}}\right)/ \approx$$
where the disjoint union (endowed with the disjoint union topology)
runs through all equivalence classes of stable marked graphs $(\Gamma,h)$ of genus $g$
and $\approx$ is the equivalence relation generated by the lattice-preserving linear maps~$L_\iota$.
\end{defi}

\begin{prop}\label{p:tgp-is-isf}
The topological space $\Tgp$ is an ideal stacky fan with cells $C^0_{(\Gamma,h)}$.  It parametrizes marked pure tropical curves of genus~$g$.
\end{prop}

\begin{proof}
To prove the first sentence, we only need to prove that the map
\begin{equation}\label{tgp-bij}
\coprod\C^0_{(\Gamma,h)}\to \left(\coprod {\C_{(\Gamma,h)}}\right){/ \approx}
\end{equation}
is bijective, where the disjoint unions
run through all stable marked graphs $(\Gamma,h)$ of genus $g$; then the rest follows from the definition of $\Tgp$.  It is injective because the linear maps in Lemma~\ref{L:specialization-ideal} never identify two different points in the relative interiors of cones.  It is surjective by (\ref{celtrop-dec}).  Finally, from the bijection (\ref{tgp-bij}) and Remark~\ref{r:no-stackiness}, it follows that $\Tgp$ parametrizes marked pure tropical curves of genus~$g$.
\end{proof}

\subsection{Homeomorphism with Outer space}\label{S:Outerspace}

 As mentioned above, for a fixed genus $g$, the pure tropical Teichm\"uller space $\Tgp$ and Culler-Vogtmann's Outer space $X_g$ parametrize the same objects, provided that we do not normalize the volume of the metric graphs as in \cite{cullervogtmann}.
We now recall the definition of~$X_g$.

Let $(\G,l)$ be a metric graph; we will often write $\G$ for short, and regard $\G$ as a metric space just as in Definition~\ref{metric-graph}.  A {\em loop} in $\G$ is the image of a map $S^1 \ra\G$.  It is {\em immersed} if it is a local homeomorphism onto its image.  The {\em length} of an immersed loop $\al$ is the sum of the lengths of the edges it traverses, counted with multiplicity.  The length of a non-immersed loop $\al$ is defined to be the length of the unique up to homotopy immersed loop homotopic to $\al$.

Recall that $F_g$ denotes the free group on $g$ letters, the graph $R_g$ denotes the rose with $g$ petals, and that we fix an identification between $F_g$ and $\pi_1(R_g,v)$, where $v$ is the unique vertex of $R_g$.  So a word $w\in F_g$ determines, up to homotopy, a loop $\lambda(w)$ in $R_g$; furthermore, two words determine the same loop, again up to homotopy, if and only if they are in the same conjugacy class.

 Let $\C$ denote the set of conjugacy classes of words in $F_g$. Then any marked metric graph $(\Gamma, l, h)$ of genus $g$  determines a real valued function  $h_{\C}(\Gamma, l, h)$ on $\C$ which assigns to each word $w$ the length of the unique immersed loop in $(\Gamma,l)$ which is homotopic to $h(\lambda(w))$. As noted above, the definition of $h_\C(\Gamma, l, h)$ does not depend on the equivalence class of $w$.

\begin{factdefi} \cite{cullervogtmann}\label{F:x-g}
Outer space $X_g$ is defined to be the set of equivalence classes of stable marked metric graphs $(\G,l,h)$.  The map $h_\C\colon X_g\ra \R^\C$ defined above is an injection by \cite{cullermorgan}, and we equip $X_g\subset \R^\C$ with the subspace topology. $X_g$ has a natural decomposition into a disjoint union of open simplicial cones consisting of
equivalence classes of marked metric graphs  having the same combinatorial type.
\end{factdefi}

Notice that since we do not normalize the volume of the graphs as Culler and Vogtmann do, we get an embedding  of  $X_g$ in $\mathbb R^\C$ rather than in $\R\P^{\C}$ as in \cite{cullermorgan}, and we get that $X_g$ is a union of simplicial cones as opposed to simplices.

The following fact is ``folklore'' in geometric group theory.  A proof appears in \cite{guirardel-levitt}; in fact the result there pertains to a wider class of deformation spaces.

\begin{prop}\cite[Proposition~5.4]{guirardel-levitt} The subspace topology on Outer space coincides with the simplicial topology (obtained from gluing together the above simplicial cones along shared boundaries.)
\end{prop}

As an immediate corollary, we have:
\begin{cor}\label{T:tgp-is-xg}
The pure tropical Teichm\"uller space $\Tgp$, with topology described in Definition~\ref{Tgp}, is homeomorphic to Outer space $X_g$, with the subspace topology.
\end{cor}

\begin{remark}\label{Xgidealfan}
Note that $X_g$ itself can be regarded as the support of an ideal fan $\Sigma = \{\sigma_i\}$ whose cones are in bijection with stable marked graphs of genus $g$ \cite{cullervogtmann}.  As described in Remark~\ref{r:infinite}, we may therefore associate an ideal stacky fan $X_{g,\Sigma}$ to $X_g$: roughly speaking, we pull the space apart into its ideal cones, impose a lattice on each cone separately, and then reglue.  Corollary~\ref{T:tgp-is-xg} then implies that $X_{g,\Sigma}$ and $\Tgp$ are isomorphic ideal stacky fans.
\end{remark}

\subsection{Tropical Teichm\"uller space}\label{ss:tg}

Having constructed $\Tgp$, we will now construct a space, which we call tropical Teichm\"uller space and denote by $\Tg$, by replacing the ideal cones in $\Tgp$ with closed cones.  Roughly speaking, we allow edge lengths to go to zero; we then have to do a little work to define a marking on such a ``pseudo-metric'' graph and when two markings should be considered the~same.

\begin{defi}
\label{mark-graph}
A \emph{weighted graph} is a pair $(\Gamma,w)$ consisting of a
graph $\Gamma$ and a function $w:V(\Gamma)\to \Z_{\geq 0}$, called
the weight function.
A weighted graph is called \emph{stable} if any vertex $v$ of weight zero
(i.e. such that $w(v)=0$) has valence $\val(v)\geq3$.
The total weight of $(\Gamma,w)$ is
$$|w|:=\sum_{v \in V(\Gamma)} w(v),$$
and the genus of $(\Gamma,w)$ is defined to be
$$g(\Gamma,w):=g(\Gamma)+|w|.$$
We will denote by $\un{0}$ the identically zero weight function.
\end{defi}

\begin{defi}\label{trop-curves}
  A {\em tropical curve} of genus $g$ is a triple $(\Gamma, w, l)$, where $(\Gamma,w)$ is a stable weighted graph of genus $g$ and $l:E(\Gamma)\to \R_{>0}$ is a length function on the edges of $\Gamma$.
\end{defi}

Note that pure tropical curves in the sense of Definition \ref{metric-graph} are exactly the tropical curves with total weight zero.


In order to endow weighted graphs with a marking, we will use the strategy of A. Omini and L. Caporaso in \cite{RR} that treats a weight of $w>0$ at a vertex as a bouquet of $w$ loops attached to that vertex.

\begin{defi}\label{D:virt-graph}
Let $(\Gamma,w)$ be a weighted graph. Then the {\em virtual graph} of $(\Gamma,w)$ is the graph $\Gamma^w$ obtained from $\Gamma$ by attaching to every vertex $v$ exactly $w(v)$ loops
, which will be called the {\em virtual loops} of $\Gamma^w$.
\end{defi}

\begin{defi}
A {\em pseudo-metric graph} is a pair $(\Gamma,l)$ where $\Gamma$ is a graph and $l:E(\Gamma)\to \R_{\geq 0}$ is a length function on the edges which is allowed to vanish only on loop edges of $\Gamma$. A pseudo-metric graph $(\Gamma, l)$ is said to be stable if the underlying graph $\Gamma$ is stable.

Given a tropical curve $(\Gamma,w,l)$, we associate to it the pseudo-metric graph $(\Gamma^w,l^w)$, where $l^w$ is the length function obtained by extending $l$ to be equal to $0$ on the virtual loops of $\Gamma^w$.
\end{defi}

\noindent Notice that the virtual graph $\Gamma^w$ is an unweighted graph of the same genus as $(\Gamma,w)$, so the correspondence above gives a bijection between tropical curves and stable pseudo-metric graphs of fixed genus.

\begin{defi}\label{d:marked-weighted}
Let $(\Gamma,w)$ be a weighted graph of genus $g$.
Then a {\em marking} on $(\Gamma,w)$ is a marking on $\Gamma^w$, that is,~a homotopy equivalence $h\colon R_g\ra\Gamma^w$.

\end{defi}

We will shortly define two markings of weighted graphs to be equivalent if, roughly speaking, they are the same up to homotopy and up to interchanging sequences of virtual loops based at the same vertex.  To make this definition precise requires some notation, as follows.

Recall that $R_g$ is the graph consisting of $g$ loops $\gamma_1,\ldots, \gamma_g$ at a vertex $v$, and let $h\colon R_g\ra \G$ be any  map.  By moving the image of $v$ appropriately and then ``pulling tight'' along each $\gamma_i$, we see that $h$ is homotopic to some $\tilde{h}\colon R_g\ra \G$ that sends $v$ to a point of  $\G$ which is not a vertex (any point of $\G$ of zero weight will do) and furthermore immerses the interior of each $\gamma_i$.  

\begin{defi}\label{d:marked-weighted-equ}
Two markings $h,h'\colon R_g\ra\Gamma^w$ of $(\Gamma,w)$ are equivalent if, after homotoping $h$ and $h'$, we have
\begin{enumerate}[(i)]
	\item $h(v)=h'(v)$ is a point of  $\G$ which is not a vertex,
	\item $h$ and $h'$ are immersions on the interiors of each $\gamma_i$,  and
	\item for each $i=1,\ldots,g$, the directed loop $h'(\gamma_i)$ is obtained, up to homotopy fixing the basepoint, from the directed loop $h(\gamma_i)$ by repeatedly replacing sequences of virtual loops at a vertex with other sequences of virtual loops at that vertex.  (These sequences of virtual loops are allowed to be empty: in other words, one may add or remove virtual loops at a vertex.)
\end{enumerate}
Two marked weighted graphs $(\G,w,h)$ and $(\G',w',h')$ are equivalent if there is an isomorphism $\G^w\ra(\G')^{w'}$ of 1-complexes that takes virtual loops to virtual loops and sends $h$ into the equivalence class of $h'$.  Two marked metric weighted graphs $(\G,w,l,h)$ and $(\G',w',l',h')$ are {\em equivalent} if the underlying marked weighted graphs $(\G,w,h)$ and $(\G',w',h')$ are equivalent via an isomorphism that respects the length function $l$ on the edges of $\G$.
\end{defi}
\noindent As in the previous section, we are interested in marked weighted graphs (resp.~marked metric weighted graphs) up to equivalence, and will often drop the phrase  ``equivalence class of'' for simplicity.

\begin{remark}
The reason for conditions (i) and (ii) above is to prohibit adding a virtual loop at a vertex $v$ to a path that passes through $v$ and then immediately backtracks; roughly speaking, this is because such a path can be homotoped to one that does not backtrack through $v$, but a path that uses a virtual loop at $v$ cannot be homotoped in that way.
\end{remark}

We can now define marked tropical curves.

\begin{defi}\label{D:trop-curves}
A {\em marked tropical curve} $(C,h)$ of genus $g$ is a tropical curve $C=(\G,w,l)$ of genus $g$ together with a marking $h:R_g\to \Gamma^w$ of the underlying weighted graph $(\G,w)$, up to equivalence.

Given a marked tropical curve $(C,h)=(\G,w,l,h)$, we call the  marked weighted graph $(\G,w,h)$
the combinatorial type of the tropical curve.
\end{defi}

\begin{defi}\label{D:simplicial-cones-2}
Given a stable marked weighted graph $(\Gamma, w, h)$ of genus $g$, fix a numbering on its set of edges $E(\Gamma)$.  Write
$$\C_{(\Gamma,w,h)}^0 := \R^{|E(\Gamma)|}_{>0} \qquad\text{and}\qquad \ov \C_{(\Gamma,w,h)}:=\mathbb R^{|E(\Gamma)|}_{\ge0}$$
for the open and closed simplicial cones, respectively, in $\mathbb R^{|E(\Gamma)|}$.
\end{defi}

\begin{lemma}\label{l:no-stackiness-2}
Let $(\G,w,h)$ be (an equivalence class of) a stable marked  weighted graph.  Then the points of $\C_{(\Gamma,w,h)}^0$ are in bijection with the set of (equivalence classes of) marked metric weighted graphs $(\Gamma,w,l,h)$ for some length function $l$, or in other words with marked tropical curves of combinatorial type $(\G,w,h)$.
\end{lemma}

\begin{proof}
The argument is just a slight strengthening of Remark~\ref{r:no-stackiness}.  A point in $\mathbb R^{|E(\Gamma)|}_{>0}$ determines a length function $l\colon E(\G)\ra\R_{>0}$.  Suppose $\psi$ is a nontrivial isometry of $(\G,l,w)$.  That is, $\psi$ is an isomorphism of $\G^w$ restricting to an isomorphism of $\G$ that respects $l$ and $w$, such that the induced permutation $p\colon E(\G)\ra E(\G)$ is nontrivial.  We want to show that $\psi h$ is not equivalent to $h$.

Just as in Remark~\ref{r:no-stackiness}, if $\psi$ moves some loop in $\G$ to a nonhomotopic loop, we are done.  So we may assume
$\psi$ fixes every loop in $\G$.  Furthermore, $\psi$ must fix every vertex of positive weight, otherwise $\psi h$ would differ from $h$ on the virtual loops.   Then $\psi$ restricts to the identity on $\G$, contradiction.
 \end{proof}

Generalizing Definition \ref{D:spec-pure}, we now introduce specializations of weighted graphs and marked weighted graphs.

\begin{defi}\label{d:specialization}
Let $(\Gamma,w)$ and $(\Gamma',w')$ be weighted graphs. We say that $(\Gamma,w)$  specializes to $(\Gamma',w')$, and we write $(\Gamma,w)\rightsquigarrow (\Gamma',w')$, if $\Gamma'$ is obtained from $\Gamma$ by collapsing some of its edges and if the weight function of the specialized curve changes according to the following rule: if we contract a loop $e$ around a vertex $v$ then we increase the weight of $v$ by one; if we contract an edge $e$ between two distinct vertices $v_1$ and $v_2$ then we obtain a new vertex with weight equal to $w(v_1)+w(v_2)$.
\end{defi}

Note that if  $(\Gamma,w)\rightsquigarrow (\Gamma',w')$ then $(\Gamma,w)$ and $(\Gamma', w')$ have the same genus; if moreover $(\Gamma,w)$ is stable then $(\Gamma', w')$ is stable.

\begin{defi}
Let $(\Gamma,w,h)$ be a marked weighted graph. Consider a specialization  $(\Gamma,w)\rightsquigarrow (\Gamma',w')$ and call $S\subseteq E(\Gamma)$ the subset consisting of the edges of $\Gamma$ that are contracted in order to obtain $\Gamma'$. Now pick any spanning forest of $S$ and contract the edges in it. This operation yields a marking $h'$ on $(\G')^{w'}$ since it contracts no cycles. Picking a different spanning forest produces an equivalent marking, so we have a marked weighted graph $(\G',w',h')$ that is well-defined up to equivalence.  We say that $(\G',w',h')$ is a {\em specialization} of $(\G,w,h)$ in this situation and we write $(\Gamma,w,h)\rightsquigarrow(\Gamma',w',h')$
\end{defi}

Just as in \S\ref{ss:tgp}, a specialization $(\Gamma,w,h)\rightsquigarrow(\Gamma',w',h')$ yields an obvious inclusion $\iota\colon E(\Gamma')\ra E(\Gamma)$ that induces, in turn, a lattice-preserving linear map $L_\iota\colon \ov\C_{(\Gamma',w',h')}\hookrightarrow\ov\C_{(\Gamma,w,h)}$ sending $\ov\C_{(\Gamma',w',h')}$ to a face of $\ov\C_{(\Gamma,w,h)}$.  Using these linear maps, given a stable marked weighted graph $(\Gamma,w,h)$, we have a natural identification of sets
\begin{equation}\label{celtrop-dec-2}
\ov \C_{(\Gamma,w,h)}=\coprod_{(\Gamma,w,h)\rightsquigarrow(\Gamma',w',h')}\C^0_{(\Gamma',w',h')},
\end{equation}
Summarizing, we have:

\begin{lemma}\label{L:spec-ideal-2}
Let $(\Gamma, w,h)$ be (an equivalence class of) a stable marked weighted graph of genus $g$. Then the faces of ${\ov\C_{(\Gamma,w,h)}}$ are in bijective correspondence with cones $\ov\C_{(\Gamma',w',h')}$, where  $(\Gamma',w',h')$ is a specialization of $(\Gamma, w,h)$.  The identification is given by the lattice-preserving linear map $L_\iota\colon \ov \C_{(\Gamma',w',h')}\hookrightarrow\ov \C_{(\Gamma,w,h)}$.
\end{lemma}

\begin{defi}\label{Tg}
The tropical Teichm\"uller space of genus $g$ is the topological space (with respect to the quotient
topology)
$$\Tg:=\left(\coprod {\ov\C_{(\Gamma,w,h)}}\right)/ \approx$$
where the disjoint union (endowed with the disjoint union topology)
runs over all stable marked weighted graphs $(\Gamma,w,h)$ of genus $g$,
and $\approx\,$ is the equivalence relation generated by the lattice-preserving linear maps $L_\iota$.
\end{defi}

\begin{prop}\label{p:tg-is-isf}
The topological space $\Tg$ is a stacky fan with cells $\C^0_{(\Gamma,w,h)}$.  It parametrizes marked tropical curves $(\Gamma,w,l,h)$ of genus~$g$.
\end{prop}

\begin{proof}
The proof is analogous to the one for Proposition~\ref{p:tgp-is-isf}. For the first part, we only need to prove that the map
\begin{equation}\label{tg-bij}
\coprod\C^0_{(\Gamma,w,h)}\ra \left(\coprod {\ov\C_{(\Gamma,w,h)}}\right){/ \approx}
\end{equation}
is bijective.  It is injective because the linear maps in Lemma~\ref{L:spec-ideal-2} never identify two different points in the relative interiors of cones.  It is surjective by (\ref{celtrop-dec-2}).  Finally, from the bijection (\ref{tg-bij}) and Lemma~\ref{l:no-stackiness-2}, it follows that $\Tg$ parametrizes  marked tropical curves of genus $g$.
\end{proof}

\subsection{Action of outer automorphism group on Outer Space}\label{outergroup}

The group $$\Out(F_g)=\Aut(F_g)/\Inn(F_g)$$ acts on $X_g$ by changing the markings. More precisely, given $\alpha \in\Aut(F_g)$, consider its geometric realization $\alpha_R:R_g\to R_g$, i.e. the homeomorphism, unique up to homotopy, that fixes the vertex $v$ of $R_g$ and such that the induced automorphism $(\alpha_R)_*:F_g=\pi_1(R_g,v)\to \pi_1(R_g,v)=F_g$ is equal to $\alpha\in \Aut(F_g)$.
 Then define $(\Gamma,l,h)\cdot\alpha=(\Gamma,l,h\circ \alpha_R)$. It is easy to see that this action is well defined and that inner automorphisms act trivially, so we get an action from $\Out(F_g)$ on $X_g$.  We may equally well view this as an action on $\Tgp$, since $\Tgp\cong X_g$ by Corollary~\ref{T:tgp-is-xg}.
Note that the stabilizer of any marked graph $(\Gamma,l,h)$ is equal to the group of isometries of $(\Gamma,l)$, and thus it is finite.

The action of $\Out(F_g)$ on $\Tgp$ extends to an action of $\Out(F_g)$ on $\Tg$, again by changing the marking.

\begin{prop}\label{p:outer-is-admissible}
The $\Out(F_g)$-actions on $\Tgp$ and $\Tg$ are admissible with respect to the ideal stacky fan structures of the latter spaces, in the sense of Definition~\ref{d:admissible-action}.
\end{prop}
\begin{proof}
The action of $\al\in\Out(F_g)$ preserves the tropical curve underlying the marked tropical curve, so in each case, each cone $C_{(\G,w,h)}$ of $\Tgp$ is mapped to some cone $C_{(\G,w,h')}$ via a map that is clearly a lattice-preserving isomorphism (and similarly for the cones in $\Tg$.) Thus the conditions of Definition~\ref{d:admissible-action} are satisfied.
\end{proof}

\comment{
\begin{lemma}
Outer space is locally finite for the action of $\Out(F_g)$. (\textbf{Reference for this fact?})
\end{lemma}
}

\section{Moduli space of tropical curves}\label{s:mtrg}



The moduli space $\Mtrg$ of tropical curves of genus $g$ was first constructed in \cite{BMV} (see also \cite{capsurvey} and \cite{chan}). We start by reviewing this construction, adapting it to the new definition \ref{d:sf} of stacky
fans, which is slightly different from the definitions of stacky fans given in \cite{BMV} and \cite{chan} (see however
Remark \ref{R:old-def}). We also introduce the open subspace $\Mp\subset \Mtrg$ parametrizing pure tropical curves of genus $g$ and show that it is an ideal stacky fan.


\begin{defi}\label{D:simplicial-cones-3}
Given a stable weighted graph $(\Gamma, w)$ of genus $g$, fix a numbering on its set of edges $E(\Gamma)$.  Write
$$\C_{(\Gamma,w)}^0 := \R^{|E(\Gamma)|}_{>0} \qquad\text{and}\qquad \ov \C_{(\Gamma,w)}:=\mathbb R^{|E(\Gamma)|}_{\ge0}$$
for the open and closed simplicial cones, respectively, in $\mathbb R^{|E(\Gamma)|}$.
\end{defi}

Let $(\G,w)$ be a stable, weighted graph and let $\Aut(\Gamma,w)$ be its automorphism group.
Then $\Aut(\Gamma,w)$ acts on $\C_{(\Gamma,w)}^0$ and on its closure $\ov \C_{(\Gamma,w)}$ by permuting
coordinates.

\begin{remark}\label{r:stackiness}
The points of $\C_{(\Gamma,w)}^0/\Aut(\Gamma,w)$ are in bijection with
the set of tropical curves of combinatorial type $(\Gamma, w)$, i.e. tropical curves of the form $(\Gamma,w,l)$ for some length function $l$.
\end{remark}

\begin{remark}\label{R:Aut-Lin}
For each $\alpha\in\Aut(\Gamma,w)$, denote by $L_{\alpha}: \ov\C_{(\Gamma,w)}\to \ov\C_{(\Gamma,w)} $ the induced lattice-preserving linear map.
Note that the points of $\C_{(\Gamma,w)}^0/\Aut(\Gamma,w)$ are in bijection with the points in the quotient $\C_{(\Gamma,w)}^0/\sim$, where $\sim$ is the equivalence relation generated by the lattice preserving linear maps $L_\alpha|_{\C_{(\Gamma,w)}^0}$, $\alpha\in\Aut(\Gamma,w)$.
\end{remark}

Let now $(\Gamma,w)\rightsquigarrow (\Gamma',w')$ be a specialization of weighted graphs as defined in Definition \ref{d:specialization}.
As in  \S\ref{ss:tg},  a specialization $(\Gamma,w)\rightsquigarrow(\Gamma',w')$ yields an obvious inclusion $\iota\colon E(\Gamma')\ra E(\Gamma)$ that induces, in turn, a lattice-preserving linear map $L_\iota\colon \ov\C_{(\Gamma',w')}\hookrightarrow\ov\C_{(\Gamma,w)}$ sending $\ov\C_{(\Gamma',w')}$ to a face of $\ov\C_{(\Gamma,w)}$.
Analogously to (\ref{celtrop-dec-2}), using these linear maps, given a stable weighted graph $(\Gamma,w)$, we have a natural identification of sets
\begin{equation}\label{celtrop-dec3}
\ov \C_{(\Gamma,w)}=\coprod_{(\Gamma,w)\rightsquigarrow(\Gamma',w')}\C^0_{(\Gamma',w')},
\end{equation}
Similarly to Lemma \ref{L:spec-ideal-2}, we have the following


\begin{lemma}\label{L:spec-ideal-3}
Let $(\Gamma, w)$ be a stable graph of genus $g$. Then the faces of $\ov\C_{(\Gamma,w)}$ are in bijective correspondence with the cones $\ov\C_{(\Gamma',w')}$, where  $(\Gamma',w')$ is a specialization of $(\Gamma, w)$.  The identification is given by the lattice-preserving linear map $L_\iota\colon \ov \C_{(\Gamma',w')}\hookrightarrow\ov \C_{(\Gamma,w)}$. 
\end{lemma}

\begin{defi}\label{Mg}
The moduli space of tropical curves of genus $g$ is the topological space (with respect to the quotient
topology)
$$\Mtrg:=\left(\coprod {\ov\C_{(\Gamma,w)}}\right)/ \approx$$
where the disjoint union (endowed with the disjoint union topology)
runs over all stable weighted graphs $(\Gamma,w)$ of genus $g$,
and $\approx\,$ is the equivalence relation generated by the lattice-preserving linear maps $L_\iota$ and $L_{\alpha}$.
\end{defi}

Let $\Gamma$ be a stable graph of genus $g$, which can view as the stable weighted graph $(\Gamma,\un 0)$ with zero weight function.
 Denote by $\C_{\Gamma}$ the ideal subcone of $\ov\C_{\Gamma}:=\ov\C_{(\Gamma,\un 0)}$ whose ideal faces correspond (via Lemma \ref{L:spec-ideal-3}) to the specializations $(\Gamma, \un 0)\rightsquigarrow (\Gamma',\un 0)$  such that the weight function remains identically zero.
 In other words, the points of  $\C_{\Gamma}$ correspond to pure tropical curves whose combinatorial type is a specialization of $\Gamma$. Clearly $\C_{\Gamma}$ contains the open cone $\C^0_{\Gamma}:=\C^0_{(\Gamma, \un 0)}$.

\begin{defi}\label{Mgpure}
The moduli space of pure tropical curves of genus $g$ is the topological space (with respect to the quotient
topology)
$$\Mp:=\left(\coprod {\C_{\Gamma}}\right)/ \approx$$
where the disjoint union (endowed with the disjoint union topology)
runs over all stable graphs $\Gamma$ of genus $g$,
and $\approx\,$ is the equivalence relation generated by the lattice-preserving linear maps $L_\iota$ and $L_{\alpha}$.
\end{defi}

\begin{prop}\label{p:Mg-is-isf}
The topological space $\Mtrg$ (resp. $\Mp$) is a stacky fan (resp. an ideal stacky fan) with cells $C^0_{(\Gamma,w)}/\!\sim$ (resp. $C^0_\Gamma/\!\sim$) as $(\Gamma,w)$ (resp. $\Gamma$) runs over all stable weighted graphs $(\Gamma,w)$ (resp.~stable graphs $\Gamma$) of genus $g$.  It parametrizes tropical curves $(\Gamma,w,l)$  (resp.~pure tropical curves $(\Gamma,l)$) of genus~$g$.
\end{prop}

\begin{proof}
We will prove the statement for $\Mtrg$, the case of $\Mp$ being analogous.
For the first part we only need to prove that
\begin{equation}\label{Mg-bij}
\coprod(\C^0_{(\Gamma,w)}{/ \sim})\ra \left(\coprod {\ov\C_{(\Gamma,w)}}\right){/ \approx}
\end{equation}
is bijective, where the disjoint union runs over all stable weighted graphs $(\Gamma,w)$ of genus $g$; then the rest
follows from the definition of $\Mtrg$.

The map \eqref{Mg-bij} is surjective by \eqref{celtrop-dec3}. In order to prove that it is injective,
consider two points $x$ and $y$ lying in $\C^0_{(\Gamma,w)}$ and $\C^0_{(\Gamma',w')}$, respectively. Then, since the maps $L_\iota$ associated to specializations of weighted graphs never identify two different points in the relative interior of cones, $x$ and $y$ are identified in  $\left(\coprod {\ov\C_{(\Gamma,w)}}\right){/ \approx}$ if and only if $(\Gamma,w)=(\Gamma',w')$ and there is $\alpha\in\Aut(\Gamma,w)$ such that $L_{\alpha} (x)=y$. The injectivity of  \eqref{Mg-bij} now follows from Remark \ref{R:Aut-Lin}.

Finally, combining the bijection \eqref{Mg-bij} with Remarks \ref{r:stackiness} and \ref{R:Aut-Lin},
it follows that $\Mtrg$ parametrizes tropical curves of genus $g$.
\end{proof}

Recall that the outer automorphism group  $\Out(F_g)$ acts on $\Tgp$ and on $\Tg$ by changing the markings (see \S\ref{outergroup}) and that the action is admissible with respect to their ideal stacky fan structures (see Proposition \ref{p:outer-is-admissible}).
We can then form the stratified quotients of $\Tgp /\!/ \Out(F_g)$ and $\Tg /\!/ \Out(F_g)$ which, by Proposition \ref{p:isstackyfan} are again ideal stacky fans endowed with a map of ideal stacky fans from $\Tgp$ and $\Tg$, respectively. The next result shows that these stratified quotients are isomorphic, as ideal stacky fans, to $\Mp$ and $\Mtrg$, respectively.

\begin{prop}\label{P:Mg-quot}
 The moduli space of pure tropical curves (resp. tropical curves) $\Mp$ (resp. $\Mtrg$) is the stratified quotient, hence global quotient, of $\Tgp $ (resp. $\Tg$) modulo the action of $\Out (F_g)$.
\end{prop}

\begin{proof}
We will again prove the statement only for $\Mtrg$, since the case of $\Mp$ is analogous. Note that the cones of $\Tg$ fall into orbits under the admissible action of $\Out(F_g)$ precisely according to the isomorphism type of $(\G,w)$, since the $\Out(F_g)$ acts transitively on the markings of a weighted graph $(\G,w)$.  So the cells of $\Tg/\!/\Out(F_g)$ are indeed of the form $\ov C_{(\G,w)}$ as $(\G,w)$ ranges over combinatorial types of genus $g$ tropical curves.

Next, the linear maps in Definition~\ref{d:sq2}, along which the cones $\ov C_{(\G,w)}$ are glued, are of the form
$$\ov C_{(\G',w',h')} \longrightarrow\ov C_{(\G',w',h'')}\hooklongrightarrow\ov C_{(\G,w,h)},$$
where the first map, induced by an $\Out(F_g)$-action as in Proposition~\ref{p:outer-is-admissible}, changes the marking, and the second map,  induced by a specialization as in Lemma~\ref{L:spec-ideal-2}, is an inclusion of faces.  It follows that for every isomorphism $(\G',w')$ with a specialization of $(\G,w)$, we obtain a gluing map $\ov C_{(\G',w')} \hookrightarrow \ov C_{(\G,w)}$, and furthermore, all gluing maps in the stratified quotient are of this form.  Such maps correspond precisely to specializations of weighted graphs and to automorphisms of weighted graphs: the former occur when the inclusion $\ov C_{(\G',w',h'')}\hooklongrightarrow\ov C_{(\G,w,h)}$ of faces in the composition above is proper, while the latter occur when the inclusion is non-proper, that is, bijective.




\end{proof}

\begin{cor}\label{C:quot-Xg}
We have a homeomorphism
$$\Mp\cong X_g/\Out(F_g).$$
\end{cor}
\begin{proof}
Combine  Proposition \ref{P:Mg-quot} with Corollary~\ref{T:tgp-is-xg}.

\end{proof}


\section{Tropical Siegel space and moduli space of tropical abelian varieties}\label{S:admi-deco}

The aim of this section is to introduce the moduli space of tropical abelian varieties and a cover of it, which we
call the tropical Siegel space, parametrizing marked tropical abelian varieties.
Our construction will depend upon the choice of an admissible decomposition of the cone of positive definite
quadratic forms, which we now review. Our treatment of this and of several other important definitions in this section follows \cite{BMV}.

\subsection{Admissible decompositions}\label{S:adm-deco}

We denote by $\R^{\binom{g+1}{2}}$ the vector space of quadratic forms
in $\R^g$ (identified with $g\times g$ symmetric matrices with
coefficients in $\R$) and by $\O$ the cone in $\R^{\binom{g+1}{2}}$ of positive
definite quadratic forms. The closure $\ov{\O}$ of $\O$ inside $\R^{\binom{g+1}{2}}$ is the cone
of positive semi-definite quadratic forms. We will be working with a partial closure of the cone $\O$ inside $\ov{\O}$, the so called rational closure of $\O$ (see \cite[Sec. 8]{NamT}).

\begin{defi}\label{D:rat-forms}
A positive definite quadratic form $Q$ is said to be \emph{rational} if the null space $\Null(Q)$ of $Q$
(i.e. the biggest subvector space $V$ of $\R^g$ such that $Q$ restricted to $V$ is identically zero)
admits a basis with elements in $\Q^g$.

We will denote by $\Ort$ the cone of rational positive semi-definite quadratic forms.
\end{defi}

The group $\GL_g(\Z)$ acts on the vector space $\R^{\binom{g+1}{2}}$ of quadratic forms
via the usual law $h\cdot Q:= h Q h^t$, where $h\in \GL_g(\Z)$ and $h^t$ is the transpose matrix.
Clearly, the cones $\O$ and $\Ort$ are preserved by the action of $\GL_g(\Z)$.

\begin{remark}\label{rat-qua}
It is well-known (see \cite[Sec. 8]{NamT}) that a positive semi-definite
quadratic form $Q$ in $\R^g$ belongs to $\Ort$ if and only if
there exists $h\in \GL_g(\Z)$ such that
$$hQh^t=
\left(\begin{array}{cc}
Q' & 0 \\
0 &  0 \\
\end{array}\right)
$$
for some positive definite quadratic form $Q'$ in $\R^{g'}$, with $0\leq g'\leq g$.
\end{remark}

The cones $\O$ and its rational closure $\Ort$ are not polyhedral. However they can be subdivided into rational polyhedral subcones in a nice way, as in the following definition (see \cite[Lemma 8.3]{NamT}).

\begin{defi}\label{decompo}
An \emph{admissible decomposition} of $\Ort$ is a collection $\Sigma=\{\sigma_{\mu}\}$ of rational polyhedral cones of
$\Ort$ such that:
\begin{enumerate}[(i)]
\item If $\sigma$ is a face of $\sigma_{\mu}\in \Sigma$ then $\sigma\in \Sigma$;
\item The intersection of two cones $\sigma_{\mu}$ and $\sigma_{\nu}$ of $\Sigma$ is a face of both cones;
\item If $\sigma_{\mu}\in \Sigma$ and $h\in \GL_g(\Z)$ then $h\cdot \sigma_{\mu}
\cdot h^t\in \Sigma$.
\item $\#\{\sigma_{\mu}\in \Sigma \mod \GL_g(\Z)\}$ is finite;
\item $\cup_{\sigma_{\mu}\in \Sigma} \sigma_{\mu}=\Ort$.
\end{enumerate}
We say that two cones $\sigma_{\mu}, \sigma_{\nu}\in \Sigma$ are equivalent if they are conjugated by an element of
$\GL_g(\Z)$. We denote by $\Sigma/\GL_g(\Z)$ the finite set of equivalence classes of cones in $\Sigma$. Given a cone $\sigma_{\mu}\in \Sigma$, we denote by $[\sigma_{\mu}]$ the equivalence class
containing~$\sigma_{\mu}$.
\end{defi}

\begin{defi}\label{D:admissDecomPositive}
Let $\Sigma=\{\sigma_{\mu}\}$ be an admissible decomposition of $\Ort$. Define $\Sigma_{|\O}$ to be the restriction of $\Sigma$ to $\O$, i.e.,
$$\Sigma_{|\O}:=\{\O\cap\sigma_{\mu}:\sigma_{\mu}\in\Sigma \mbox{ and the intersection is non-empty}\}.$$
\end{defi}

It is clear from the definition that $\cup_{\sigma_{\mu}\in \Sigma_{|\O}}\sigma_{\mu}=\O$.


\begin{prop}\label{P:comp-admdec}
Let $\Sigma$ be an admissible decomposition of $\Ort$ as in Definition \ref{decompo}. Then $\Sigma$
(resp.~$\Sigma_{|\O}$) is a $\GL_g(\Z)$-admissible decomposition of $\Ort$ (resp.~$\O$) in the sense of Definition \ref{strat-defi}.
\end{prop}
\begin{proof}
Let us first prove the statement for $\Sigma$.
The fact that $\Sigma$ is a (rational polyhedral) fan with support equal to $\Ort$ together with the fact that
the action of $\GL_g(\Z)$ on $\Ort$ permutes the cones of $\Sigma$ follows from Definition \ref{decompo}.
In order to prove that $\Sigma$ is a $\GL_g(\Z)$-admissible decomposition of $\Ort$, it remains to check that
the maps induced by the action of an element of $\GL_g(\Z)$ on the cones of $\Sigma$ are linear and lattice-preserving.
In fact, given $h\in\GL_g(\Z)$ and a cone $\sigma\in\Sigma$, then $h$ induces a linear isomorphism between $\sigma$ and $h\cdot\sigma=\sigma'$ for some $\sigma'\in\Sigma$. Moreover, a positive semi-definite matrix $A\in\sigma$ is such that $h\cdot A=hAh^t\in \Z^{\binom{g+1}{2}}$ if and only if $A\in\Z^{\binom{g+1}{2}}$ and the result follows.

Let us now prove the statement for $\Sigma_{|\O}$. We begin by showing that $\Ort\setminus \O$ is a union of cones of $\Sigma$.
Let $A\in\Ort\setminus \O$ and assume that $A$ lies in the relative interior of a cone $\sigma$. It suffices to show that the whole cone $\sigma$ is contained in $\Ort\setminus\O$. Suppose that $\sigma$ is generated by matrices $A_1,\dots, A_k\in\Ort$.
Since $A\in \Ort\setminus \O$ then there exists $x\in \R^n$ such that $xAx^t=0$. Since
$A$ can be written as a strictly positive linear combination of all the $A_i$'s (because it is in the relative interior of
$\sigma$), then we have that $xA_ix^t=0$ for $i=1,\dots, k$. This implies that any element $B\in \sigma$ satisfies $xBx^t=0$, hence that $\sigma\in\Ort\setminus\O$, as required.
We deduce that $\Sigma_{|\O}$ is the (rational polyhedral) ideal fan obtained from $\Sigma$ from removing the cones which are entirely contained in $\Ort\setminus \O$. It is clear that the support of $\Sigma_{|\O}$ is equal to $\O$.
The fact that $\Sigma_{|\O}$ is a $\GL_g(\Z)$-admissible decomposition of $\O$ follows now from the analogous fact for $\Sigma$
together with the fact that the action of $\GL_g(\Z)$ on $\Ort$ preserves $\O$ (see Remark \ref{rat-qua}).

\end{proof}

The following result will be very useful in what follows.


\begin{lemma}\label{L:loc-fin}
For any admissible decomposition $\Sigma$ of $\Ort$, the restriction $\Sigma_{|\O}$ is a locally finite ideal fan.
\end{lemma}
\begin{proof}
Let $x\in \O$ and consider a closed polyhedral subcone $C\subset \O$ containing $x$ in its interior (clearly there are plenty of such subcones).
Take now any   cone $\sigma_{\mu} \in \Sigma$. From the classical theory of Siegel sets (see \cite[Chap. II.4]{AMRT}), it follows that the set
$$G_{\sigma_{\mu}}:=\{h\in \GL_g(\Z)\: :\: h\cdot \sigma_{\mu}\cap C\neq \emptyset  \}\subset \GL_g(\Z)$$
is finite (see \cite[Corollary at page 116]{AMRT}). Using this and the fact that there are only finitely many $\GL_g(\Z)$-equivalence classes of cones of $\Sigma$ (see Definition \ref{decompo}), we conclude that
$$\#\{\sigma_{\mu}\in \Sigma \: :\: \sigma_{\mu}\cap C\neq \emptyset\}< \infty .$$
This shows that $\Sigma_{|\O}$ is a locally finite ideal fan with support equal to $\O$.

\end{proof}

\subsection{Examples of admissible decompositions}\label{S:exa-deco}

A priori, there could exist infinitely many admissible decompositions of $\Ort$. However,
as far as we know, only three admissible decompositions are known for every integer $g$
(see \cite[Chap. 8]{NamT} and the references therein), namely:
\begin{enumerate}[(i)]
\item The perfect cone decomposition (also known as the first Voronoi
decomposition), which was first introduced in \cite{Vor};
\item The 2nd Voronoi decomposition (also known as the L-type decomposition), which was first introduced in \cite{Vor};
\item The central cone decomposition, which was introduced in \cite{Koe}.
\end{enumerate}
Each of them plays a significant (and different) role in the theory of the
toroidal compactifications of the moduli space of principally
polarized abelian varieties (see \cite{Igu}, \cite{alex1}, \cite{SB}).

\begin{example}
If $g=2$ then all the above three admissible decompositions coincide.
In Figure \ref{VorFig} we illustrate a section of the
$3$-dimensional cone $\Omega_2^{\rm rt}$, where we represent just some of the infinite
cones of the admissible decompositions. Note that, for $g=2$, there is only one
$\GL_g(\Z)$-equivalence class of maximal dimensional cones, namely
the principal cone $\prin^0$ (see \cite[Sec. (8.10)]{NamT}).

\begin{figure}[h]
\begin{center}
\scalebox{.5}{\epsfig{file=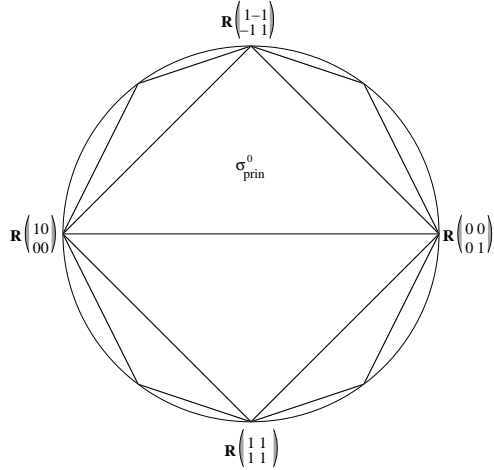}}
\end{center}
\caption{A section of $\Omega_2^{\rm rt}$ and its admissible
decomposition.  In this case, the perfect cone decomposition, 2nd Voronoi decomposition, and central cone decomposition are all the same.}
\label{VorFig}
\end{figure}
\end{example}

In this paper, we will consider the perfect cone decomposition and the 2nd Voronoi decomposition
since these behave well with respect to the period mapping (see Section \ref{S:periodmap}).


\subsubsection{The perfect cone decomposition $\Per$}\label{S:Perfect}
\indent

In this subsection, we review the definition and the main properties of the perfect cone admissible decomposition
(see \cite{Vor} for more details and proofs, or \cite[Sec. (8.8)]{NamT} for a summary).

Consider the function $\mu:\O\to \R_{> 0}$ defined by
$$\mu(Q):=\min_{\xi\in \Z^g\setminus \{0\}} Q(\xi).$$
It can be checked that, for any $Q\in \O$, the set
$$M(Q):=\{\xi\in \Z^g\: : \: Q(\xi)=\mu(Q) \} $$
is finite and non-empty.
For any $\xi\in M(Q)$, consider the rank one quadratic form $\xi\cdot \xi^t\in \Ort$.
We denote by $\sigma[Q]$ the rational polyhedral subcone of $\Ort$ given by the convex hull of the rank one forms
obtained from elements of $M(Q)$, i.e.
$$\sigma[Q]:=\R_{\geq 0}\langle \xi\cdot \xi^t\rangle_{\xi\in M(Q)}.
$$
One of the main results of \cite{Vor} is the following

\begin{fact}[Voronoi]
\label{F:main-Per}
The set of cones
$$\Per:= \{ \sigma[Q] \: :\:  Q\in \O \}\cup \{0\} $$
yields an admissible decomposition of $\Ort$, known as the {\em perfect cone decomposition}.
\end{fact}

The quadratic forms $Q$ such that $\sigma[Q]$ has maximal dimension $\binom{g+1}{2}$ are called \emph{perfect},
hence the name of this admissible decomposition. The interested reader is referred to  \cite{Mar} for more details on perfect forms.


\begin{example}\label{E:Per-g2}
 Let us compute $\Per$ in the case $g=2$ (compare with Figure \ref{VorFig}).
Let
$R_{12} = \left(\begin{matrix}1&-1\\-1&1\end{matrix}\right)$,
$R_{13} = \left(\begin{matrix}1&0\\0&0\end{matrix}\right)$,
$R_{23} = \left(\begin{matrix}0&0\\0&1\end{matrix}\right).$
Then, up to $\GL_g(\Z)$-equivalence, an easy computation shows that the unique non-zero cones in $\Per$ are
\begin{align*}
 &\sigma\left[\left(\begin{matrix}1&1/2\\1/2&1\end{matrix}\right)\right] = \R_{\geq 0} \langle R_{12}, R_{13}, R_{23} \rangle=\left\{\left(\begin{matrix}a+c&-c\\-c&b+c\end{matrix}\right) \: : \: a, b, c \geq 0\right\} , \\
 &\sigma\left[\left(\begin{matrix}1&\lambda\\\lambda&1\end{matrix}\right)\right] = \R_{\geq 0} \langle R_{13}, R_{23} \rangle=\left\{\left(\begin{matrix}a&0\\0&b\end{matrix}\right) \: : \: a, b \geq 0\right\} \text{ for any } -1/2<\lambda <1/2, \\
 &\sigma\left[\left(\begin{matrix}1&\lambda\\\lambda&\mu\end{matrix}\right)\right] = \R_{\geq 0} \langle R_{13} \rangle=
 \left\{\left(\begin{matrix}a&0\\0&0\end{matrix}\right) \: : \: a \geq 0\right\} \text{ for any } \mu>\max\{1, \lambda^2, \pm 2\lambda\}. \\
\end{align*}

\end{example}

\subsubsection{The 2nd Voronoi decomposition $\Voro$}\label{S:Voro}
\indent

In this subsection, we review the definition and main properties of the 2nd Voronoi admissible decomposition
(see \cite{Vor}, \cite[Chap. 9(A)]{NamT} or \cite[Chap. 2]{Val} for more details and proofs).

The Voronoi decomposition is based on the Delone subdivision $\Del(Q)$ associated to a quadratic form $Q\in \Ort$.

\begin{defi}
  Given $Q \in \Ort$, consider the map $l_Q : \Z^g \to \Z^g \times \R$ sending $x \in \Z^g$ to $(x,Q(x))$. View the image of $l_Q$ as an infinite set of points in $\R^{g+1}$, one above each point in $\Z^g$, and consider the convex hull of these points. The lower faces of the convex hull
 can now be projected to $\R^g$ by the map $\pi:\R^{g+1}\to \R^g$ that forgets the last coordinate. This produces an infinite $\Z^g$-periodic polyhedral subdivision of $\R^g$, called the {\em Delone subdivision} of $Q$ and denoted $\Del(Q)$.
\end{defi}

It can be checked that if $Q$ has rank $g'$ with $0\leq g'\leq g$ then $\Del(Q)$ is a subdivision consisting of polyhedra
such that the maximal linear subspace contained in them has dimension $g-g'$. In particular, $Q$ is positive definite
if and only if $\Del(Q)$ is made of polytopes, i.e. bounded polyhedra.

Now, we group together quadratic forms in $\Ort$ according to the Delone subdivisions that they yield.

\begin{defi}\label{D:open-secondary}
  Given a Delone subdivision $D$ (induced by some $Q_0\in \Ort$), let
  \[ \sigma_D^0 = \{ Q \in \Ort : \Del(Q) = D \}. \]
\end{defi}

It can be checked  that the set $\sigma_D^0$ is a relatively open (i.e. open in its linear span) rational polyhedral cone in $\Ort$.
Let $\sigma_D$ denote the Euclidean closure of $\sigma_D^0$ in $\R^{\binom{g+1}{2}}$, so $\sigma_D$ is a closed rational polyhedral cone and $\sigma_D^0$ is its relative interior. We call $\sigma_D$ the {\em secondary cone} of $D$.

Now, the action of the group $GL_g(\Z)$ on $\R^g$ induces an action of $GL_g(\Z)$ on the set of Delone subdivisions: given a Delone subdivision $D$ and an element $h\in \GL_g(\Z)$, denote by $h\cdot D$ the Delone subdivision given by the action of $h$ on $D$. Moreover, $\GL_g(\Z)$ acts naturally on the set
of secondary cones $\{\sigma_D:D \text{ is a Delone subdivision of } \R^g\}$ in such a way that
$$h\cdot \sigma_D:=\{hQh^t \: : \: Q\in \sigma_D \}=\sigma_{h\cdot D}.$$
Another of the main results of \cite{Vor} is the following

\begin{fact}[Voronoi]
\label{F:main-Vor}
The set of secondary cones
$$\Voro:= \{ \sigma_D : D \text{ is a Delone subdivision of } \R^g \} $$
yields an admissible decomposition of $\Ort$, known as the {\em second Voronoi decomposition}.
\end{fact}

The cones of $\Voro$ having maximal dimension $\binom{g+1}{2}$ are those of the form $\sigma_D$ for $D$ a Delone subdivision which is a triangulation, i.e. such that $D$ consists only of simplices (see \cite[Sec. 2.4]{Val}).
We refer the reader to \cite{MV} for a comparison between the 2nd Voronoi decomposition $\Sigma_V$ and
the perfect decomposition $\Sigma_P$.

\begin{example}
\label{e:Atr2}
  Let us compute $\Voro$ in the case $g=2$ (compare with Figure \ref{VorFig} and with Example \ref{E:Per-g2}).  Combining the taxonomies in \cite[Sec. 4.1, Sec. 4.2]{Val}, we may choose four representatives $D_1, D_2, D_3, D_4$ for $\GL_g(\Z)$-orbits of Delone subdivisions as in Figure \ref{f:subdiv}, where we have depicted the part of the Delone subdivision that fits inside the unit cube in $\R^2$.

\begin{figure}[h]%
\includegraphics[width=3in]{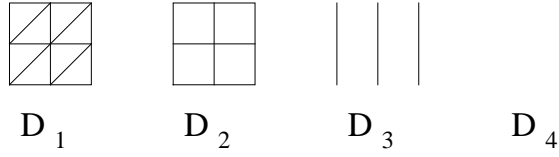}%
\caption{Delone subdivisions for $g=2$ (up to $\GL_g(\Z)$-equivalence).}%
\label{f:subdiv}%
\end{figure}
\noindent We can describe the corresponding secondary cones as follows. Let
$R_{12} = \left(\begin{matrix}1&-1\\-1&1\end{matrix}\right)$,
$R_{13} = \left(\begin{matrix}1&0\\0&0\end{matrix}\right)$,
$R_{23} = \left(\begin{matrix}0&0\\0&1\end{matrix}\right)$ as in Example \ref{E:Per-g2}.
Then
\begin{align*}
 \sigma_{D_1} &= \R_{\geq 0} \langle R_{12}, R_{13}, R_{23} \rangle=\left\{\left(\begin{matrix}a+c&-c\\-c&b+c\end{matrix}\right) \: : \: a, b, c \geq 0\right\} , \\
 \sigma_{D_2} &= \R_{\geq 0} \langle R_{13}, R_{23} \rangle=\left\{\left(\begin{matrix}a&0\\0&b\end{matrix}\right) \: : \: a, b \geq 0\right\} , \\
 \sigma_{D_3} &= \R_{\geq 0} \langle R_{13} \rangle=
 \left\{\left(\begin{matrix}a&0\\0&0\end{matrix}\right) \: : \: a \geq 0\right\} , \\
 \sigma_{D_4} &= \{0\}.
\end{align*}
\end{example}

\subsection{Tropical Siegel space}\label{S:tropSiegel}

The aim of this subsection is to introduce the tropical Siegel space, which parametrizes marked tropical abelian
varieties, whose definition we now introduce.

\begin{defi}\label{D:trop-abvar}
\noindent
\begin{enumerate}[(i)]
\item \label{D:trop-abvar1} A \emph{tropical p.p.} (= principally polarized) \emph{abelian variety} $A$ of dimension $g$ is a pair $(V/\Lambda, Q)$ consisting of a $g$-dimensional real torus $V/\Lambda$ (so that $V$ is a $g$-dimensional real vector space and $\Lambda\subset V$ is a full-dimensional lattice)  and  $Q$ is a positive semi-definite quadratic form on $V$ such that the null space $\Null(Q)$ of $Q$ is defined over $\Lambda\otimes \Q$, i.e. it admits a basis with elements in $\Lambda\otimes \Q$.

A tropical p.p.~abelian variety $A=(V/\Lambda,Q)$ is said to be {\em pure} if $Q$ is positive definite.

\item \label{D:trop-abvar2} A {\em marking} on a p.p.~abelian variety $A=(V/\Lambda,Q)$ is an isomorphism
$\phi:\R^g/\Z^g\stackrel{\cong}{\to} V/\Lambda$ of real tori, or equivalently a linear isomorphism from $\R^g$ onto $V$ sending $\Z^g$ isomorphically onto $\Lambda$.

We say that $(A,\phi)=(V/\Lambda,Q,\phi)$ is a {\em marked tropical p.p.~abelian variety}.

\end{enumerate}
\end{defi}

The above definition of tropical p.p.~abelian varieties is due to \cite{BMV}, generalizing slightly the definition of  \cite{mz}, where only pure tropical p.p.~abelian varieties are considered.

Indeed, marked tropical p.p.~abelian varieties up to isomorphism are the same thing as positive semi-definite quadratic forms, 
as observed in the following

\begin{remark}\label{R:equiv-abvar}
\noindent
\begin{enumerate}[(i)]
\item \label{R:equiv-abvar1} Every marked tropical p.p.~abelian variety $(V/\Lambda, Q, \phi)$ is uniquely determined by the quadratic form $\phi^*(Q)$ on $\R^g$ obtained by pulling back the quadratic form $Q$ on $V$ via the marking $\phi$.


\item \label{R:equiv-abvar2} Every tropical p.p.~abelian variety $A=(V/\Lambda,Q)$ is isomorphic to a tropical p.p.~abelian variety of the form $(\R^g/\Z^g,Q')$.
Moreover, we have that $(\R^g/\Z^g,Q)\cong (\R^g/\Z^g,Q')$ if and
only if there exists $h\in \GL_g(\Z)$ such that $Q'=h Q h^t$, i.e. if and only if
$Q$ and $Q'$ are arithmetically equivalent.
\end{enumerate}
\end{remark}

Given any admissible decomposition $\Sigma$ of $\Ort$, we consider the ideal stacky fan associated to
the fan $\Sigma$, which we view as the tropical analogue of the classical Siegel space.

\begin{defi}\label{D:trop-Sieg}
Let $\Sigma$ be an admissible decomposition of $\Ort$. We denote by $\Hg$ (resp.~$\Hgp$) the stacky fan (resp.~ideal stacky fan)
associated to the fan $\Sigma$ (resp.~the ideal fan $\Sigma_{|\O}$) according to Remark \ref{r:asso-fan}
and we call it the {\em tropical Siegel space}
(resp.~the {\em pure tropical Siegel space}) associated with $\Sigma$.
\end{defi}

Any (pure) tropical Siegel space parametrizes marked (pure) tropical p.p.~abelian varieties, as shown in the following

\begin{prop}\label{P:Siegel}
Fix an admissible decomposition $\Sigma$ of $\Ort$.
\begin{enumerate}[(i)]
\item \label{P:Siegel1} The tropical Siegel space  $\Hg$ (resp.~the pure tropical Siegel space $\Hgp$) is a stacky fan
(resp.~an ideal stacky fan) parametrizing marked tropical p.p.~abelian varieties (resp.~marked pure
tropical p.p.~abelian varieties) of dimension~$g$.
\item \label{P:Siegel2} The map
\begin{equation}\label{E:Siegel-quadra}
\begin{aligned}
\Phi: \Hg & \longrightarrow  \Ort \\
(V/\Lambda,Q,\phi)& \mapsto \phi^*(Q)
\end{aligned}
\end{equation}
is a continuous bijection which restricts to a homeomorphism between $\Hgp$ and $\O$.
\end{enumerate}
\end{prop}
\begin{proof}
The map $\Phi$ is continuous by the proof of Lemma \ref{L:locally-finite} and bijective by Remark \ref{R:equiv-abvar}\eqref{R:equiv-abvar1}. Since $\Sigma_{|\O}$ is a locally finite ideal fan by Lemma \ref{L:loc-fin},
Lemma \ref{L:locally-finite} implies that the restriction of $\Phi$ to $\Hgp$ induces a homeomorphism between $\Hgp$ and $\Phi(\Hgp)=\O$.
\end{proof}

\subsection{Moduli space of tropical abelian varieties}\label{S:trop-abvar}

The aim of this subsection is to introduce the moduli space of tropical p.p.~abelian varieties of fixed dimension~$g$.

\begin{defi}\label{D:trop-Ag}
Let $\Sigma$ be an admissible decomposition of $\Ort$. We denote by $\Ag$ (resp.~$\Agp$) the stacky fan (resp.~ideal stacky fan)
obtained as the stratified quotient of $\Ort$ (resp.~$\O$) with respect to the $\GL_g(\Z)$-admissible decomposition
$\Sigma$ (resp.~$\Sigma_{|\O}$) as in Definition \ref{d:sq}.
\end{defi}

Consider the action of $\GL_g(\Z)$ on $\Hg$ given by changing the markings. More precisely, an
element $h\in \GL_g(\Z)$ acts on $\Hg$ by sending $(A,\phi)\in \Hg$ into $(A, \phi\circ \ov{h})$ where $\ov h:\R^g/\Z^g\stackrel{\cong}{\to} \R^g/\Z^g$ is the isomorphism induced by the linear map $h$.
Clearly the ideal stacky subfan $\Hgp\subseteq \Hg$ is preserved by the action of $\GL_g(\Z)$.
Observe that the above defined action of $\GL_g(\Z)$ on $\Hg$ makes the map $\Phi$ of \eqref{E:Siegel-quadra} equivariant with respect to the natural action of $\GL_g(\Z)$ on $\Ort$ (see \S\ref{S:adm-deco}).

\begin{lemma}\label{L:action-GL}
Fix an admissible decomposition $\Sigma$ of $\Ort$. The action of $\GL_g(\Z)$ on the stacky fan $\Hg$ (resp. on the ideal stacky fan $\Hgp$) defined above is admissible.
\end{lemma}
\begin{proof}
This follows by combining Proposition \ref{P:comp-admdec} and Remark \ref{r:spec-cases}.
\end{proof}

We can now prove that the space $\Ag$ (resp.~$\Agp$) is a moduli space for tropical p.p.~abelian varieties (resp.~pure tropical p.p. abelian varieties) and it is a quotient of the tropical Siegel space $\Hg$ (resp.~the pure tropical Siegel space $\Hgp$) by the group $\GL_g(\Z)$.

\begin{prop}\label{P:Siegel-Ag}
Fix an admissible decomposition $\Sigma$ of $\Ort$.
\begin{enumerate}[(i)]
\item \label{P:Siegel-Ag1} $\Ag$ (resp.~$\Agp$) is a stacky fan (resp.~ideal stacky fan) parametrizing tropical p.p.~abelian varieties (resp.~pure tropical p.p.~abelian varieties) of dimension $g$.
\item  \label{P:Siegel-Ag2} There is a morphism of stacky fans $\Hg\to \Ag$ (resp.~of ideal stacky fans $\Hgp\to \Agp$) which realizes $\Ag$
(resp.~$\Agp$) as the stratified quotient, hence global quotient, of $\Hg$ (resp.~$\Hgp$) by the group $\GL_g(\Z)$.

\item \label{P:Siegel-Ag3} $\Agp$ is homeomorphic to the quotient of $\O$ by the group $\GL_g(\Z)$.
\end{enumerate}
\end{prop}
\begin{proof}
The actions of $\GL_g(\Z)$ on the stacky fan $\Hg$ and on the ideal stacky fan $\Hgp$ are
admissible by Lemma \ref{L:action-GL}.
Moreover, the stratified quotient of $\Hg$ (resp.~$\Hgp$) by the group $\GL_g(\Z)$ is clearly isomorphic to the stacky fan $\Ag$ (resp.~to the ideal stacky fan $\Agp$)
again by Remark \ref{r:spec-cases}. Therefore, part \eqref{P:Siegel-Ag2} follows by combining Proposition \ref{p:isstackyfan} and Proposition \ref{p:strat-is-global}.

Part \eqref{P:Siegel-Ag1} follows now from part \eqref{P:Siegel-Ag2} together with Proposition \ref{P:Siegel}\eqref{P:Siegel1} and Remark \ref{R:equiv-abvar}\eqref{R:equiv-abvar2}.

Part \eqref{P:Siegel-Ag3} follows from part \eqref{P:Siegel-Ag2} together with Proposition \ref{P:Siegel}\eqref{P:Siegel2}.

\end{proof}

\section{The tropical period map}\label{S:periodmap}

The aim of this section is to define the tropical period map from the (pure) tropical Teichm\"uller space to the (pure) tropical Siegel space and to show that it descends to the tropical Torelli map studied in \cite{BMV} and \cite{chan}.
The period map will send a marked tropical curve into its marked tropical Jacobian, that we are now going to describe.

\subsection{(Marked) tropical Jacobians}\label{S:trop-Jac}

The tropical Jacobian of a tropical curve was defined in \cite[Sec. 5.1]{BMV}, following the earlier definition of
Mikhalkin-Zharkov in \cite[Sec. 6]{mz} in the case of pure tropical curves.

\begin{defi}\label{D:trop-Jac}
Let $C=(\Gamma,w,l)$ be a tropical curve of genus $g$. The {\em tropical Jacobian} (or simply the Jacobian) of $C$ is the tropical p.p.~abelian variety of dimension $g$
$$J(C):=\left(\frac{H_1(\Gamma,\R)\oplus \R^{|w|}}{H_1(\Gamma,\Z)\oplus \Z^{|w|}}, Q_C\right)
$$
where the quadratic form $Q_C$ is identically zero on $\R^{|w|}$ and it is given on $H_1(\Gamma,\R)$ by
\begin{equation}\label{E:QC}
Q_C\left(\sum_{e\in E(\Gamma)}\alpha_e\cdot e\right)=\sum_{e\in E(\Gamma)}\alpha_e^2 \cdot l(e).
\end{equation}
\end{defi}
\noindent
Note that a tropical curve $C$ is pure (i.e. $w=\un 0$) if and only if its tropical Jacobian $J(C)$ is  pure
(i.e.~$Q_C$ is positive definite).

Corresponding to any marking of a tropical curve is a marking of its Jacobian.

\begin{defi}\label{D:mark-Jac}
Let $(C,h)=(\Gamma,w,l, h)$ be a marked tropical curve of genus $g$. The {\em marked tropical Jacobian} (or simply the marked Jacobian) of $(C,h)$ is the marked tropical p.p.~abelian variety of dimension~$g$
$$J(C,h)=(J(C),\phi_h),$$
where $J(C)$ is the Jacobian of $C$ and
$$\phi_h:\frac{\R^g}{\Z^g} \stackrel{\cong}{\longrightarrow}\frac{H_1(\Gamma,\R)\oplus \R^{|w|}}{H_1(\Gamma,\Z)\oplus \Z^{|w|}}$$
is the marking of $J(C)$ which is induced by the linear isomorphism $$\R^g=H_1(R_g,\R)\xrightarrow[h_*]{\cong} H_1(\Gamma^w,\R)\cong H_1(\Gamma,\R)\oplus \R^{|w|}$$
where the first isomorphism is induced by the marking $h\colon R_g\to \Gamma^w$ and the second isomorphism is induced by the canonical map $\Gamma^w\to \Gamma$ that contracts the virtual loops of $\Gamma^w$ (see Definition \ref{D:virt-graph}).
\end{defi}
It is easy to see that the above defined marking $\phi_h$ on $J(C)$ depends only on the equivalence class of $h$ (see Definition \ref{d:marked-weighted-equ}); therefore, the above definition is well posed. Moreover, it is clear that a marked tropical curve $(C,h)$ is pure if and only if its marked tropical Jacobian $(J(C),\phi_h)$ is pure.

\subsection{The tropical period map}\label{S:per-map}

The tropical period map is defined as it follows.

\begin{lemmadefi}\label{D:periodmap}
The tropical period map is the continuous map
$$\begin{aligned}
{\mathcal P}_g^{tr}: \Tg & \longrightarrow \Ort \\
(C,h)& \mapsto \phi_h^*(Q_C).
\end{aligned}$$
\end{lemmadefi}
\begin{proof}
We have to prove that the map ${\mathcal P}_g^{tr}$ is continuous. According to the Definition \ref{Tg} of the tropical Teichm\"uller
space $\Tg$, it is enough to show that the restriction of ${\mathcal P}_g^{tr}$ to the cone $\ov\C_{(\Gamma,w,h)}$, for each stable marked graph $(\Gamma,w,h)$ of genus $g$, is continuous. This follows from the fact that  the quadratic form $Q_C$ on $H_1(\Gamma,\R)$ depends continuously on the lengths $l\in \R_{\geq 0}^{|E(\Gamma)|}$, as is clear from formula \eqref{E:QC}.

\end{proof}

\begin{remark}\label{R:pure-to-pure}
By the observation before Definition \ref{D:trop-Jac}, we have that
$$({\mathcal P}_g^{tr})^{-1}(\O)=\Tgp.$$
\end{remark}

Recall that the tropical Teichm\"uller space $\Tg$ has a natural stacky fan structure (see Proposition \ref{p:tg-is-isf}). On the other hand, the stacky fan structure of $\Ort$ depends on the choice of an admissible decomposition $\Sigma$ of
$\Ort$ (see Definition \ref{D:trop-Sieg}). Some admissible decompositions of $\Ort$ are compatible with the tropical period map
${\mathcal P}_g^{tr}$ in the following sense.

\begin{defi}\label{D:compat-deco}
An admissible decomposition $\Sigma$ of $\Ort$ (see Definition \ref{decompo}) is said to be {\em compatible with the tropical period map} if for each cell $\C^0_{(\Gamma,w,h)}$ of $\Tg$ there exists a cone $\sigma\in \Sigma$ such that
$${\mathcal P}_g^{tr}(\ov\C_{(\Gamma,w,h)})\subseteq \sigma.$$
\end{defi}

Indeed, the two admissible decompositions that we have described in Section \ref{S:exa-deco}, namely the perfect cone
decomposition and the 2nd Voronoi decomposition, are compatible with the tropical period map.

\begin{fact}[Mumford-Namikawa, Alexeev-Brunyate]\label{F:compa-P-V}
The perfect cone decomposition $\Sigma_P$ and the 2nd Voronoi decomposition $\Sigma_V$ are compatible with the tropical period map.
\end{fact}
\begin{proof}
The fact that $\Sigma_V$ is compatible with the tropical period map is due to Namikawa \cite{nam1} (who says that Mumford
was aware of it); the fact that $\Sigma_P$ is compatible with the tropical period map is due to Alexeev-Brunyate \cite{AB}.
\end{proof}

\begin{remark}\label{R:central-cone}
It is known that the central cone decomposition (studied in \cite{Koe} and \cite{Igu}) is not compatible with the tropical period map if $g\geq 9$ (see \cite{AB}), while it is compatible with the tropical period map if $g\leq 8$ (see \cite{many}).
\end{remark}

Given an admissible decomposition $\Sigma$ of $\Ort$ that is compatible with the tropical period map ${\mathcal P}_g^{tr}$, we can lift
${\mathcal P}_g^{tr}$ to a map of stacky fans with codomain the tropical Siegel space $\Hg$ associated to $\Sigma$ (see Definition \ref{D:trop-Sieg}).

\begin{thm}\label{T:period-S}
Let $\Sigma$ be an admissible decomposition of $\Ort$ that is compatible with the tropical period map in the sense of Definition \ref{D:compat-deco}. Then there exists a map of stacky fans, which we call the {\em $\Sigma$-period map}:
$$\begin{aligned}
\Pg: \Tg & \longrightarrow \Hg \\
(C,h)& \mapsto (J(C),\phi_h)
\end{aligned}$$
such that:
\begin{enumerate}[(i)]
\item \label{T:per1} The composition of \, $\Pg$ with the continuous bijection $\Phi:\Hg\to \Ort$ (see Proposition \ref{P:Siegel}\eqref{P:Siegel2}) is the tropical period map ${\mathcal P}_g^{tr}$ of Lemma-Definition \ref{D:periodmap}.
\item \label{T:per2} $\Pg$ is {\em equivariant} with respect to the homomorphism of groups
$$\Ab: \Out(F_g)\to \Out(\Z^g)=\Aut(\Z^g)=\GL_g(\Z)$$
induced by the abelianization homomorphism $F_g\to F_g^{\rm ab}=\Z^g$, and the admissible actions of $\Out(F_g)$ on $\Tg$ (see Proposition \ref{p:outer-is-admissible}) and of $\GL_g(\Z)$ on $\Hg$ (see Lemma \ref{L:action-GL}).
\item \label{T:per2bis}
We have a commutative diagram of stacky fans
\begin{equation}\label{E:diag-Tor1}
\xymatrix{
\Tg \ar[r]^{\Pg} \ar[d] &   \Hg \ar[d]\\
\Mtrg \ar[r]^{\tg} & \Ag
}
\end{equation}
where the left vertical map is the (stratified) quotient by $\Out(F_g)$, the right vertical arrow
is the (stratified) quotient by $\GL_g(\Z)$, and the map $\tg$, called the {\em tropical Torelli map}
with respect to $\Sigma$, sends a tropical curve $C$ into its tropical Jacobian $J(C)$.
\item \label{T:per3} The restriction of the diagram \eqref{E:diag-Tor1} to the pure moduli spaces is independent of the choice of $\Sigma$ and it can be identified with the commutative diagram
\begin{equation}\label{E:diag-Tor2}
\xymatrix{
X_g \ar[r]^{\Pgp} \ar[d] &   \O \ar[d]\\
X_g/\Out(F_g) \ar[r]^{\tgp} & \O/\GL_g(\Z)
}
\end{equation}
where $\Pgp$ is the continuous map (called the {\em pure tropical period map}):
$$\begin{aligned}
\Pgp: X_g& \longrightarrow \O \\
(C,h)& \mapsto \phi_h^*(Q_C).
\end{aligned}$$
and $\tgp$ is the continuous map (called the {\em pure tropical Torelli map}) induced from $\Pgp$ by quotienting the domain by $\Out(F_g)$ and the codomain by $\GL_g(\Z)$.
\end{enumerate}
\end{thm}
\begin{proof}
Part \eqref{T:per1} follows from the explicit descriptions of the maps ${\mathcal P}_g^{tr}$ and $\Pg$ together with Proposition \ref{P:Siegel}\eqref{P:Siegel2}.

Let us now prove that $\Pg$ is a map of stacky fans. Since $\Sigma$ is compatible with the tropical period map by hypothesis, given a cell $\C^0_{(\Gamma,w,h)}$, we can find a cone $\sigma\in \Sigma$ such that ${\mathcal P}_g^{tr}(\ov\C_{(\Gamma,w,h)})\subseteq \sigma. $
Therefore, we get the following commutative diagram
\begin{equation}\label{E:diag-stacky}
\xymatrix{
\ov\C_{(\Gamma,w,h)}\ar[r]^(.6){{\mathcal P}_g^{tr}}\ar@{^{(}->}[d] & \sigma\ar@{^{(}->}[r] & \Ort \\
\Tg \ar[rr]_{\Pg} && \Hg \ar[u]^{\Phi}
}
\end{equation}
where, moreover, the natural map $\Phi^{-1}(\sigma)\to \sigma$ is a homeomorphism. The restriction of the map ${\mathcal P}_g^{tr}$ to $\ov\C_{(\Gamma,w,h)}$ is the restriction of an integral linear map $\R^{E(\G)} \!\!\ra \R^{\binom{g+1}{2}}$, as it follows easily from formula \eqref{E:QC}. Therefore, the above commutative diagram shows that $\Pg$ is a continuous map and that, moreover, it is a map of stacky fans.

Part \eqref{T:per2}. As explained in \S \ref{outergroup}, the class $[\alpha]\in \Out(F_g)$ of an element $\alpha\in \Aut(F_g)$ will send $(C,h)\in \Tg$ into $(C,h)\cdot [\alpha]=(C, h\circ \alpha_R)$, where $\alpha_R:R_g\to R_g$ is the geometric realization of $\alpha$, i.e. the homeomorphism of $R_g$, unique up to homotopy, that fixes the vertex $v$ of $R_g$ and such that the induced automorphism of the fundamental group $(\alpha_R)_*^{\pi_1}\in \Aut(\pi_1(R_g, v))=\Aut(F_g)$ is equal to $\alpha$.
 According to Definition \ref{D:mark-Jac}, the marking $\phi_{h\circ \alpha_R}$
of $J(C)$ induced by $ h\circ \alpha_R$ is equal to $\phi_h\circ \ov{(\alpha_R)_*^{H_1}}$ where $\ov{(\alpha_R)_*^{H_1}}:\R^g/\Z^g\stackrel{\cong}{\longrightarrow} \R^g/\Z^g$ is the isomorphism induced by the element $(\alpha_R)_*^{H_1}\in \Aut(H_1(R_g,\Z))=\Aut(\Z^g)=\GL_g(\Z)$. Since  $H_1(R_g,\Z)$ is the abelianization of $\pi_1(R_g,v)^{\rm ab}$, we get that
$$(\alpha_R)_*^{H_1}=\Ab((\alpha_R)_*^{\pi_1})=\Ab(\alpha).$$
From this equality and the definition of the action of $\GL_g(\Z)$ on $\Hg$ (see \S\ref{S:trop-abvar}), we deduce that
$$J((C,h)\cdot [\alpha])=J(C, h\circ \alpha_R)=(J(C), \phi_h\circ \ov{\Ab(\alpha)})=J(C,h)\circ \Ab(\alpha),
$$
which concludes the proof of \eqref{T:per2}.

Part \eqref{T:per2bis}: from \eqref{T:per2} it follows that the map $\Pg$ induces, by passing to the quotient, a continuous map $\tg$ from $\Tg/\Out(F_g)$, which is homeomorphic to $\Mtrg$ by Proposition \ref{P:Mg-quot}, to $\Hg/\GL_g(\Z)$, which is homeomorphic to $\Ag$ by Proposition \ref{P:Siegel-Ag}.
Moreover, these two quotients are also stratified quotients (again by Propositions \ref{P:Mg-quot} and \ref{P:Siegel-Ag}) and therefore it follows easily that the tropical Torelli map $\tg$ is also a map of stacky fans.
Since the group $\Out(F_g)$ (resp. $\GL_g(\Z)$) acts on $\Tg$ (resp. on $\Hg$) by changing the marking,
it is clear that the tropical Torelli map $\tg$ sends $C\in \Mtrg$ into $J(C)\in \Ag$.
Finally, the commutativity of the diagram \eqref{E:diag-Tor1} follows directly from  the definition of the map $\tg$.

Part \eqref{T:per3}: the map $\Pg$ sends $\Tgp$ into $\Hgp$ as it follows from \eqref{T:per1} together with Remark
\ref{R:pure-to-pure}. Moreover, since we have the homeomorphisms $\Tgp\cong X_g$ (see Proposition \ref{p:tgp-is-isf}) and $\Hgp\cong \O$ (see Proposition \ref{P:Siegel}\eqref{P:Siegel2}), it is clear, using also \eqref{T:per1}, that the restriction of $\Pg$ to $\Tgp$ coincides with the pure tropical period map $\Pgp$.
Finally, from the homeomorphisms $\Mp\cong X_g/\Out(F_g)$  (see  Corollary \ref{C:quot-Xg}) and  $\Agp\cong \O/\GL_g(\Z)$ (see Proposition \ref{P:Siegel-Ag}\eqref{P:Siegel-Ag3}), we deduce that the restriction of $\tg$ to
$\Mp$ coincides with the pure tropical Torelli map $\tgp$.
\end{proof}

According to Fact \ref{F:compa-P-V}, we can specialize the above Theorem \ref{T:period-S} to the case where $\Sigma$ is either equal to the perfect cone decomposition $\Sigma_P$ or to the 2nd Voronoi decomposition $\Sigma_V$. In particular, the tropical Torelli map ${}^{\Sigma_V}t_g$ with respect to the 2nd Voronoi decomposition was studied in detail in \cite{BMV} and \cite{chan}, to which we refer for further details.


\section{Open questions}\label{S:open}

We end this paper with the following open questions:

\begin{enumerate}

\item In \cite{BF}, M. Bestvina and M. Feighn constructed a bordification of Outer space. It would be interesting to compare their bordification of $X_g$ with our bordification $\Tg$ of $\Tgp\cong X_g$.
\footnote{Added in proof:  Lizhen Ji informed us that there is a surjective continuous map from the Bestvina-Feighn's bordification of Outer Space to $\Tg$.}

\item As the reader may have noticed, the construction of the tropical Siegel space $\Hg$ and the moduli space $\Ag$ of tropical p.p.~abelian varieties depend on the choice
of an admissible decomposition $\Sigma$ of the cone of rational positive semi-definite quadratic forms. If we restrict to the pure open subsets $\Hgp\subset \Hg$ and $\Agp\subset \Ag$
then Proposition \ref{P:Siegel}\eqref{P:Siegel2} and  Proposition \ref{P:Siegel-Ag}\eqref{P:Siegel-Ag3} give the homeomorphisms
$$\Hgp\cong \O \hspace{1cm} \text{ and } \hspace{1cm} \Agp\cong \O/\GL_g(\Z). $$
However, we don't know if the topology of $\Hg$ and of $\Ag$ depends or not on the choice of the admissible decomposition $\Sigma$.

In particular, it would be very interesting to use the results of \cite{MV} in order to compare ${}^{\Sigma_P}\mathbb{H}_g^{tr}$ and ${}^{\Sigma_V}\mathbb{H}_g^{tr}$ (resp. ${}^{\Sigma_P}\!{A}_g^{tr}$ and ${}^{\Sigma_V}\!{A}_g^{tr}$),
where $\Sigma_P$ is the perfect cone decomposition (see Subsection \ref{S:Perfect}) and $\Sigma_V$ is the 2nd Voronoi decomposition $\Sigma_V$ (see Subsection \ref{S:Voro}).
\footnote{Added in proof: Lizhen Ji informed us that, for every admissible decomposition $\Sigma$, the tropical Siegel space ${}^{\Sigma}\mathbb{H}_g^{tr}$ is homeomorphic to a Satake partial compactification of the symmetric cone $\Omega_g$. Moreover, $\Ag$ is the cone over a Satake compactification of the associated locally symmetric space of quadratic forms of determinant $1$. In particular, the topologies of ${}^{\Sigma}\mathbb{H}_g^{tr}$ and of ${}^{\Sigma}\mathbb{H}_g^{tr}$ are independent of the choice of the admissible decomposition $\Sigma$.}

\item In \cite{BMV} (based on the results of \cite{CV}), the authors described  the fibers of the tropical Torelli map ${}^{\Sigma_V}t_g:\Mtrg\to {}^{\Sigma_V}A_g^{tr}$ with respect to the 2nd Voronoi decomposition $\Sigma_V$ (clearly the same description works for any tropical Torelli map $\tg$, because all the tropical Torelli maps coincide set-theoretically). It  should  be possible to derive from the results in loc. cit. a description of the fibers of the tropical period map ${\mathcal P}_g^{tr}:\Tg\to \Ort$ (or equivalently of the $\Sigma$-period map $\Pg:\Tg\to \Hg$ for any admissible decomposition $\Sigma$ which is compatible with the tropical period map).

\item In \cite{BMV} (see also \cite{chan}), the authors give a characterization of the image of the tropical Torelli map  ${}^{\Sigma_V}t_g:\Mtrg\to {}^{\Sigma_V}A_g^{tr}$ with respect to the 2nd Voronoi decomposition $\Sigma_V$ (indeed, using \cite{MV}, a similar description can be given for the tropical Torelli map  ${}^{\Sigma_P}t_g:\Mtrg\to {}^{\Sigma_P}A_g^{tr}$ with respect to the perfect cone decomposition $\Sigma_P$). It would be interesting to derive from the results of loc. cit. a characterization of the image of the $\Sigma_V$-period map ${}^{\Sigma_V}{\mathcal P}_g:\Tg\to {}^{\Sigma_V}{\mathbb H}_g^{tr}$ and of the  $\Sigma_P$-period map ${}^{\Sigma_P}{\mathcal P}_g:\Tg\to {}^{\Sigma_P}{\mathbb H}_g^{tr}$.

\end{enumerate}


\end{document}